\documentclass[twoclumn]{article}

\usepackage{amsmath,amsfonts,amssymb,amstext,amsthm,bbm} 

\usepackage{natbib}
\usepackage{graphicx,color}
\usepackage{wrapfig}
\usepackage[font={small,it}]{caption}
\usepackage{tabularx}
\usepackage{balance}
\usepackage{url}

\newtheorem{theorem}{Theorem}[section]
\newtheorem{lemma}[theorem]{Lemma}

\newtheorem{conjecture}[theorem]{Conjecture}
\newtheorem*{lemma*}{Lemma}
\newtheorem*{remark*}{Remark}
\theoremstyle{definition}

\newtheorem{definition}[theorem]{Definition}

\newcommand{\range}[2]{\in\{#1,\ldots,n\}}
\newcommand{\vrt}{\textsc{Vert}}
\newcommand{\ideg}{\operatorname{ideg}} 
\newcommand{\bdeg}{\operatorname{bdeg}} 

\graphicspath{{./graphics/}}

\def\FULLVERSION{}
\def\FULLVERSIONPERMTABLE{}

\begin{document}

\title{Fairly Dividing a Cake after Some Parts Were Burnt in the Oven
\footnote{
A preliminary version of this paper was accepted to the AAMAS-2018 conference
\citep{SegalHalevi2018Burnt}.
The main addition in the present version is Section \ref{sec:special-case}, which proves the existence of envy-free divisions in two additional cases.
}
}

\author{Erel Segal-Halevi
\\
Ariel University, Ariel 40700, Israel
}

\maketitle

\begin{abstract}
There is a heterogeneous resource that contains both good parts and bad parts, for example, a cake with some parts burnt, a land-estate with some parts heavily taxed, or a chore with some parts fun to do. The resource has to be divided fairly among $n$ agents
with different preferences, each of whom has a personal value-density function on the resource. The value-density functions can accept any real value --- positive, negative or zero. 
Each agent should receive a connected piece and no agent should envy another agent. We prove that such a division exists for 3 agents and present preliminary positive results for larger numbers of agents.
\end{abstract}

\maketitle

\section{Introduction}
Most research works on fair division assume that the \emph{manna} (the resource to divide) is good, e.g., tasty cakes, precious jewels or fertile land-estates. A substantial minority of the works assume that the manna is bad, e.g., house-chores or night-shifts. Recently, \citet{Bogomolnaia2017Competitive} introduced the more general setting of \emph{mixed manna} --- every resource can be good for some agents and bad for others. Here are some illustrative examples. 
\begin{enumerate}
\item A cake with some parts burnt has to be divided among children. Some (like this author as a child) find the burnt parts tasty, but most children consider them bad (but still must eat what they get in order not to insult the host).
\item A land-estate has to be divided among heirs, where landowners are subject to taxation. The value of a land-plot to an heir may be either positive or negative, depending on his/her valuation of the land and tax status.
\item A house-chore such as washing the dishes has to be divided among family members. Most of them consider this bad, but some of them may view dish-washing, in some parts of the day, as a perfect relaxation after spending hours in solving mathematical problems.
\end{enumerate}
While \citet{Bogomolnaia2017Competitive} focused on dividing homogeneous resources, 
we study the  classic problem of \emph{cake-cutting} \citep{Steinhaus1948Problem} --- dividing a single heterogeneous resource.
The cake-cutting problem comes in many flavors: the cake can be one-dimensional or multi-dimensional
\citep{SegalHalevi2017Fair}; the fairness criterion can be \emph{proportionality} (each agent receives a piece he values as at least $1/n$ of the total) or \emph{envy-freeness} (each agent receives a piece he values at least as much as the piece of any other agent); the pieces can be connected or disconnected; and more. See \cite{Branzei2015Computational,Procaccia2015Cake} for recent surveys. 
All variants were studied in the good-cake setting (all agents consider every piece of cake good). Some variants were also studied in the bad-cake setting (all agents consider every piece of cake bad). So far, no variants were studied in the general mixed-cake setting. 

While all variants of the cake-cutting problem are interesting, this paper focuses on a specific variant in which (a) the cake is one-dimensional, (b) the fairness criterion is envy-freeness, (c) the pieces must be connected (see \textbf{Section \ref{sec:model}} for the formal model).

The main question of interest in this paper is:
\begin{center}
\emph{Does there exist a connected envy-free division of a mixed cake?}
\end{center}
It is known that the answer is ``yes'' both for good cakes and for bad cakes \citep{Su1999Rental}. Moreover, there are procedures for approximating such a division for any number of agents.
However, the proofs are based on a specific combinatorial structure, based on the well-known \emph{Sperner's lemma}; this structure breaks down in the mixed-cake setting, so the existing proofs are inapplicable (\textbf{Section \ref{sec:existing}}).

Working with mixed cakes requires a new, more general combinatorial structure. This structure is based on a generalization of Sperner's lemma. Based on this structure, it is possible to prove the main result  (\textbf{Section \ref{sec:mixed}}):
\begin{center}
\emph{A connected envy-free division always exists for three agents.}%
\footnote{
Division problems with 3 agents are quite common in practice. For example, according to www.pewsocialtrends.org/2015/05/07/family-size-among-mothers, 
about 25\% of mothers have 3 children. Hence, about 25\% of inheritance cases
involve division among 3 agents.
As another example, in the spliddit.org website \citep{Goldman2015Spliddit},
about 62\% of all requests for fair division of items involve 3 agents. We thank Nisarg Shah for this information.
}
\end{center}
The existence of a connected envy-free division implies that an existing approximation algorithm can be adapted to approximate such a division to any desired accuracy (\textbf{Section \ref{sec:finding}}).

Most parts of the proof are valid for any number of agents. However, there is one part which we do not know how to generalize to an arbitrary number of agents.
Recently, \citet{meunier2018envy}
presented a proof that an envy-free division exists when the number of agents is 4 or prime.
A proof sketch for the case of prime $n$, using more elementary arguments, is given in \textbf{Section \ref{sec:special-case}}.

\section{Model}
\label{sec:model}
A cake --- modeled as the interval $[0,1]$ --- has to be divided among $n$ agents. The agents are called $A_1,\ldots,A_n$ or Alice, Bob, Carl, etc. The cake should be partitioned into $n$ pairwise-disjoint intervals, $X_1,\ldots,X_n$ (some possibly empty), whose union equals the entire cake. Interval $X_i$ should be given to $A_i$ such that the division is \emph{envy-free} --- each agent weakly prefers his piece over any other agent's piece. Two models for the agents' preferences are considered. 

\paragraph{(A) Additive agents:} each agent $A_i$ has an integrable value-density function $v_i$. The value of a piece is the integral of the value-density on that piece: $V_i(X_j)=\int_{x\in X_j} v_i(x)dx$. 
Note, the value of any single point is $0$, so it is irrelevant who receives the endpoints of pieces.
A division is \emph{envy-free} if each agent believes his piece's value is at least as high as the value of any other agent's piece: $\forall i,j: V_i(X_i)\geq V_i(X_j)$. 

\paragraph{(B) Selective agents:} each agent $A_i$ has a function $s_i$ that accepts a nonempty set of pieces $X$ and returns a nonempty subset of $X$. The interpretation is that the agent ``prefers'' each of the pieces in $s_i(X)$ over all other pieces in $X$ (this implies that the agent is indifferent between the pieces in $s_i(X)$).  A division $X$ is \emph{envy-free} if each agent receives one of his preferred pieces: $\forall i: X_i\in s_i(X)$.
The preference functions should be \emph{continuous} --- any piece that is preferred for a convergent sequence of partitions is preferred for the limit partition (equivalently: for each $i,j$, the set of partitions $X$ in which $X_j\in s_i(X)$ is a closed set. See \citet{Su1999Rental}). This, again, implies that it is irrelevant who receives the endpoints of pieces.

Model (A) is more common in the cake-cutting world, while model (B) is much more general. Every additive agent is also a selective agent%
\ifdefined\FULLVERSION
, with $s_i(X):=\arg\max_{j\range{1}{n}}(V_i(X_j))$
(in each partition, the agent selects the piece or pieces with maximal value). 
\else
.
\fi
But selective agents may have non-additive valuations and even some externalities: the preference of an agent may depend on the entire set of pieces in the partition rather than just his own piece (however, the preference may not depend on which agent receives what piece; see \cite{Branzei2013Externalities} for a discussion of such externalities). 
\hskip 1cm
In the good-cake and bad-cake settings, additional assumptions are made on the agents' preferences besides continuity, as shown in Table \ref{tab:asm}. 
The present paper removes these assumptions.

\begin{table}
\small
\begin{center}
\begin{tabular}{|c|c|c|}
\hline
\textbf{Cake}
          & \textbf{Additive agents} & \textbf{Selective agents} \\ 
\hline 
{Good}
& 
$v_i(x)\geq 0 $ for every $x\in[0,1]$.
& 
\shortstack{$s_i$ always contains a\\non-empty piece.}
\\ 
\hline 
{Bad}
&
$v_i(x)\leq 0$ for every $x\in [0,1]$.
&
\shortstack{
$s_i$ always contains an 
\\empty piece, if one exists.
}
\\
\hline 
{Mixed}
&  
$v_i$ is any integrable function.
&
\shortstack{
$s_i$ is any continuous 
\\
selection function.
}
\\ 
\hline 
\end{tabular} 
\end{center}
\caption{\label{tab:asm} Assumptions in different cake-cutting models.}
\end{table}

\paragraph{Approximately-envy-free division.}
There are two ways to define an approximately envy-free division. 
(A) With additive agents, the approximation is measured in units of value: an \emph{$\epsilon$-envy-free division} is a division in which each agent believes that his piece's value is at most $\epsilon$ less than the value of any other piece: $\forall i,j: V_i(X_i)\geq V_i(X_j)-\epsilon$.  The valuations are  usually normalized such that the value of the entire cake is $1$ for all agents, so $\epsilon$ is a fraction (e.g., $1\%$ of the cake value).
\hskip 1cm
(B) With selective agents there are no numeric values, so the approximation is measured in  units of length: a \emph{$\delta$-envy-free division} is a division in which, for every
agent $A_i$, movement of the borders by at most $\delta$ results in a
division in which $A_i$ prefers his piece over any other piece. 
\ifdefined\FULLVERSION
If $\delta$ is sufficiently small
(e.g. 0.01 millimeter) 
then an $\delta$-envy-free division 
can be considered envy-free
for all practical purposes.
\fi

Unless stated otherwise, all results in this paper are valid for selective agents, therefore also for additive agents.

\section{Existing Procedures}
\label{sec:existing}
With $n=2$ agents, the classic ``I cut, you choose'' protocol produces an envy-free division whether the cake is good, bad or mixed. The fun begins at $n=3$. 
\subsection{Reduction to all-goods and all-bads} 
One might think that mixed-manna problems  could be reduced to good-manna and bad-manna ones in the following way. For each part of the resource: (a) if there is one or more agents who think it is good, then divide it among them using any known procedure for dividing goods; (b) otherwise, all agents think it is bad --- divide it among them using any known procedure for dividing bads. 

However, this simple reduction does not work when there are additional requirements besides fairness, such as economic efficiency or connectivity. 
In 
\citet{Bogomolnaia2017Competitive} the requirements are envy-freeness and Pareto-efficiency; in this paper the requirements are envy-freeness and connectivity. 
It is impossible to guarantee all three properties simultaneously \citep{Stromquist2007Pie}.
Hence the techniques and results are quite different, and no one implies the other.

\subsection{Moving-knives and approximations}
Three procedures for connected envy-free division for three \emph{additive} agents are known: \citet{Stromquist1980How}, \citet[pages 77-78]{Robertson1998CakeCutting} and \citet{Barbanel2004Cake}. All of them use one or more knives moving continuously.
They were originally designed for good cakes and later adapted to bad cakes. 
All of them crucially rely on a \emph{monotonicity} assumption: all agents weakly prefer a piece to all its subsets (in a good cake), or all agents weakly prefer a piece to all its supersets (in a bad cake). However, monotonicity does not hold with a mixed cake, so these procedures cannot be used.  
\ifdefined\FULLVERSION
See Appendix \ref{app:connected} for details and specific negative examples.
\fi

Another algorithm that does not work, but for a different reason, is the generic approximation algorithm recently presented by \citet{Branzei2017Query} for additive agents.
Their algorithm can approximate any division that is described by linear conditions; in particular, it can approximate an envy-free division, whenever such a division exists. Since an envy-free division of a mixed cake among three agents always exists (as will be proved in this paper), their algorithm can be used to find an approximation of it. The problem is the runtime complexity: while with good cakes and bad cakes their algorithm runs in time $O(n/\epsilon)$ when $\epsilon$ is the additive approximation factor, with mixed cakes the runtime complexity might be unbounded.
\ifdefined\FULLVERSION
See Appendix \ref{app:connected} for details. 
\else
Details are in an appendix in the full version. 
\fi

\subsection{Simplex of partitions}
With four or more agents, or even with three \emph{selective} agents, no moving-knives procedures are known. 
A different approach, which works for any number of selective agents, was suggested by \citet{Stromquist1980How} and further developed by \citet{Su1999Rental}. It is based on the \emph{simplex of partitions}. To present it we introduce some notation that will also be used in the rest of the paper.

$\Delta^{n-1}$ is the $(n-1)$-dimensional standard simplex --- the points $(l_1,\ldots,l_n)$ with $l_1+\cdots+l_n = 1$. Each such point  represents a cake-partition where the piece lengths are $l_1,\ldots,l_n$; see Figure \ref{fig:simplex-of-partitions}.

$[n]$ denotes the set $\{1,\ldots,n\}$.
The $n$ vertices of $\Delta^{n-1}$ are called its \emph{main vertices} and denoted by $F_j$, for $j\in [n]$. 
Each face of $\Delta^{n-1}$ is the convex hull of some 
subset of its main vertices, 
$\operatorname{conv}_{j\in J}(F_j)$ for some $J\subseteq [n]$; this face is denoted by $F_J$. E.g, the face connecting $F_1$ and $F_2$ is denoted $F_{\{1,2\}}$, or just $F_{12}$ for short.
For each $j\in[n]$, we denote $F_{-j} := F_{[n]\setminus \{j\}} = $ the face opposite to $F_j$. In all points on $F_{-j}$, the $j$-th coordinate is 0, so they represent partitions in which piece number $j$ is empty.

\begin{figure}
	\begin{center}
		\includegraphics[width=.28\columnwidth]{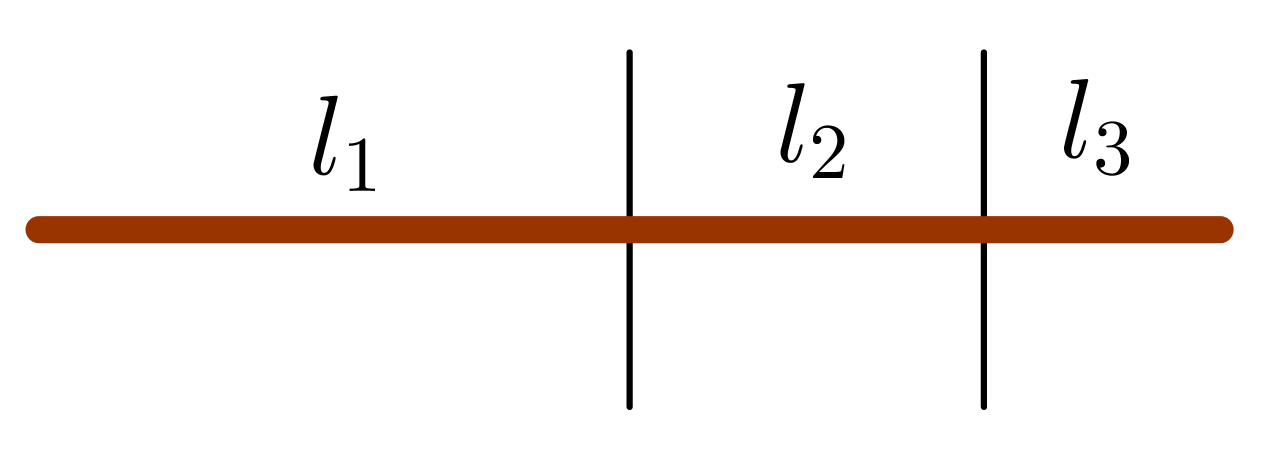}
		\hskip .04\columnwidth
		\includegraphics[width=.62\columnwidth]{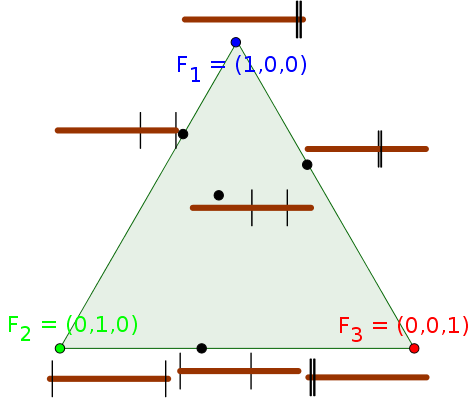}
	\end{center}
	\caption{\label{fig:simplex-of-partitions}
		\textbf{Left:} a generic partition of the cake among $n=3$ agents. $l_1+l_2+l_3=1$.
		\hskip 1cm
		\textbf{Right:} The simplex of partitions for $n=3$ agents. Each point represents a partition. Seven points are marked, and the corresponding partitions are shown.}
\end{figure}

\paragraph{Agent labelings.}
Given a partition of the cake into $n$ intervals, each agent has one or more \emph{preferred pieces}.
The preferences of agent $A_i$ can be represented by a function $L_i: \Delta^{n-1} \to 2^{[n]}$. The function $L_i$ maps each cake-partition (= a point in the standard simplex) to the set of pieces that $A_i$ prefers in this partition (= a set of labels from $[n]$). The set of preferred pieces always contains at least one label; it may contain more than one label if the agent is indifferent between two or more best pieces. This is particularly relevant in case the agent prefers an empty piece, since there are partitions in which there is more than one empty piece. If $x$ is such a partition then $L_i(x)$ contains the set of all empty pieces. See Figure \ref{fig:single-agent}.
An \emph{envy-free division} corresponds to a point $x$ in the partition-simplex where it is possible to select, for each $i$, a single label from $L_i(x)$, such that the $n$ labels are distinct.

\begin{figure}
\begin{center}
\includegraphics[width=.44\columnwidth]{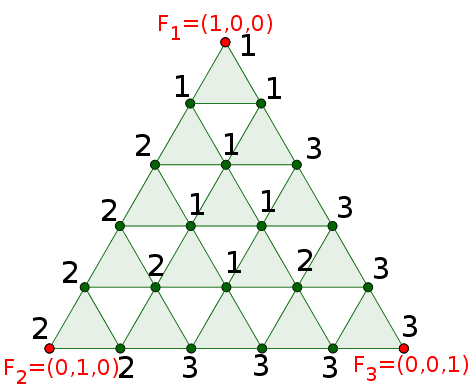}
\hskip .04\columnwidth
\includegraphics[width=.44\columnwidth]{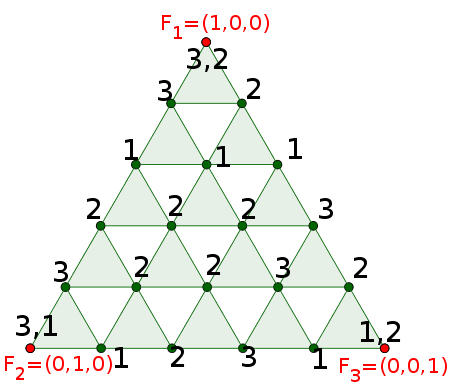}
\end{center}
\caption{\label{fig:single-agent}
Possible labeling $L_i$ of a single agent.
\textbf{Left:} the value of the entire cake is positive. Hence, in each main vertex $F_j$, the agent prefers only piece $j$, since it is the only non-empty piece. In the edges between two main vertices $F_j,F_k$, the agent prefers either $j$ or $k$.
\\
\textbf{Right:} the value of the entire cake is negative, but it contains some positive parts. In each main vertex $F_j$, the agent prefers the two empty pieces --- the two pieces that are NOT $j$. In the edges between two main vertices, all three labels may appear.
}
\end{figure}

\paragraph{Triangulations.} A triangulation of a simplex is a partition into sub-simplices satisfying some technical conditions.%
\ifdefined\FULLVERSION
\footnote{
Formally, a triangulation is determined by naming its vertices and those sets of vertices which span simplices of the triangulation, subject to the requirements that:
(i) the intersection of each pair of simplices be either empty or a simplex of the triangulation,
(ii) each face of a simplex of the triangulation is also a simplex of the triangulation,
(iii) the original simplex is the union of the simplices of the triangulation.
See \citet{Munkres1996Elements} for a formal definition and \citet{Su1999Rental} for an informal presentation.
}
\fi
~An example is shown in Figure \ref{fig:single-agent}.
We denote a triangulation by $T$, and the set of vertices in the triangulation by $\vrt(T)$. 

\begin{definition}[\textbf{Envy-free simplex}]
\label{def:ef-simplex}
Suppose we let all $n$ agents label the vertices of $T$, so we have $n$ labelings $L_i: \vrt(T)\to 2^{[n]}$
for $i\range{1}{n}$. An \emph{envy-free simplex} is a sub-simplex in $T$ with vertices $(t_1,\ldots,t_n)$, such that, for each $i\in [n]$, it is possible to select a single label from $L_i(t_i)$ such that the $n$ labels are distinct.
\end{definition}

If the diameter of each sub-simplex in $T$ is at most $\delta$, then each envy-free simplex corresponds to a $\delta$-envy-free division.
If, for every $\delta$, there is an envy-free simplex with diameter at most $\delta$, then the continuity of the preference functions $s_i$ implies the existence of an envy-free division; see \citet{Su1999Rental}.

\paragraph{Good Cakes.}
In a partition of a good cake,
there always exists a non-empty piece with a weakly-positive value, so it is always possible to assume that each agent prefers a non-empty piece. Therefore, every labeling $L_i$ satisfies \emph{Sperner's boundary condition}: every triangulation-vertex in the face $F_J$ is labeled with a label from the set $J$ (see Figure \ref{fig:single-agent}/Left).
\ifdefined\FULLVERSION
 Succinctly:
\begin{align*}
\forall i\in [n]:
\forall J\subseteq [n]:
\forall x\in F_J: L_i(x)\cap  J \neq \emptyset
\end{align*}
\fi
By Sperner's lemma, for every $i$ there is a \emph{fully-labeled simplex} --- a simplex whose $n$ vertices are labeled by $L_i$ with $n$ distinct labels.
\ifdefined\FULLVERSION
\begin{lemma}[Sperner's lemma]
	Let $T$ be a triangulation of $\Delta^{n-1}$. Let $L: \vrt(T)\to 2^{[n]}$ be a labeling. If $L$ satisfies Sperner's boundary condition, then it has an odd number of fully-labeled simplices. 
\end{lemma}
\fi

In order to get an \emph{envy-free simplex}, we combine the $n$ agent-labelings $L_1,\ldots,L_n$ to a single labeling $L^W:\vrt(T)\to 2^{[n]}$ in the following way.
Each triangulation-vertex is assigned to one of the $n$ agents, such that in each sub-simplex, each of its vertices is owned by a unique agent. See Figure \ref{fig:triangulation-ownership}/Left. Now, each vertex is labeled with the corresponding label-set of its owner: if a vertex $x$ is owned by agent $A_i$, then $L^W(x) := L_i(x)$. See Figure \ref{fig:triangulation-ownership}/Right. If all the $L_i$ satisfy Sperner's boundary condition, then the combined labeling $L^W$ also satisfies Sperner's boundary condition. Therefore, by Sperner's lemma, $L^W$ has a fully-labeled simplex. By definition of $L^W$, this simplex is an envy-free simplex \citep{Su1999Rental}.
\begin{figure}
	\begin{center}
		\includegraphics[width=.44\columnwidth]{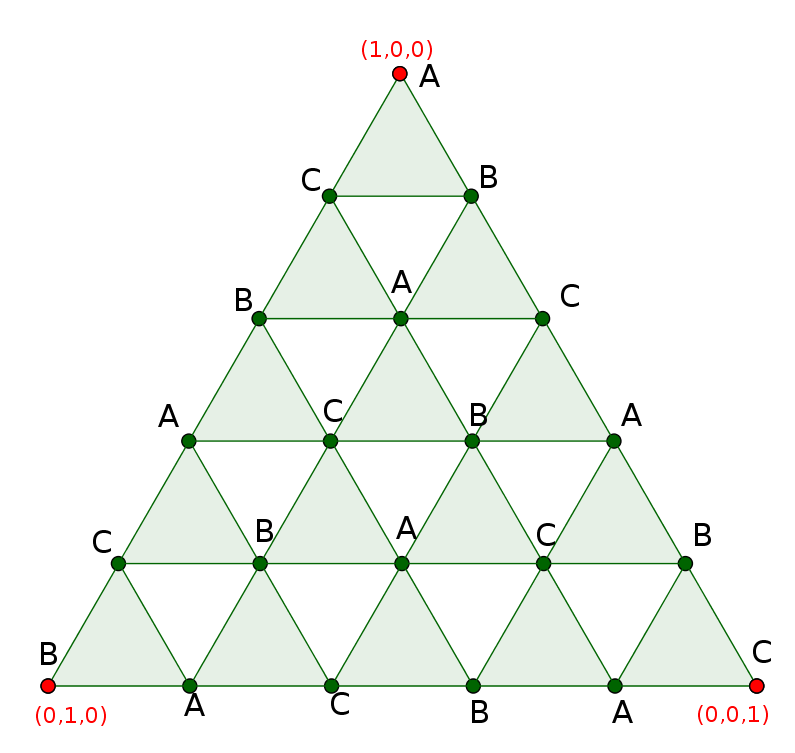}
		\hskip .05\columnwidth
		\includegraphics[width=.44\columnwidth]{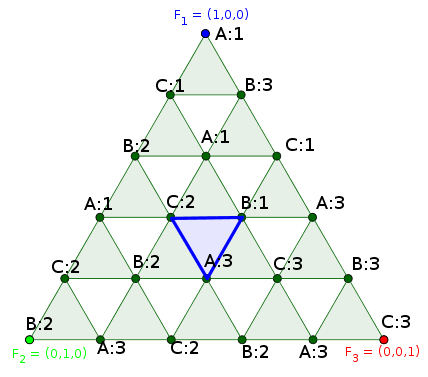}
	\end{center}
	\caption{\label{fig:triangulation-ownership} \textbf{Left}: Assignment of vertices to agents such that, in each sub-triangle, each vertex is owned by a different agent.
		\\
	\textbf{Right}: A combined labeling based on this ownership-assignment. The emphasized triangle at the center is an envy-free simplex.
}
\end{figure}

\paragraph{Bad Cakes}
In a partition of a bad cake, the values of all non-empty pieces are weakly negative, so it is always possible to assume that each agent prefers an empty piece. 
In the main vertices, there are $n-1$ empty pieces; the agent is indifferent between them, so we may label each main vertex with an arbitrary empty piece. We can always do this such that the resulting labeling satisfies Sperner's boundary condition \citep{Su1999Rental}.
\ifdefined\FULLVERSION
For example, if we label each main vertex $F_j$ by $j+1$ (modulo $n$) then the labeling satisfies Sperner's condition. 
\fi
Hence, an envy-free simplex exists.

\paragraph{Mixed Cakes}
When the value of the entire cake is negative, but the cake may contain positive pieces,
each agent may prefer in each point either an empty piece or a non-empty piece. Hence, the agent labelings no longer satisfy Sperner's boundary condition; see Figure \ref{fig:single-agent}/Right. Here, our work begins.

\begin{figure}
	\begin{center}
		\hskip 1cm
		\includegraphics[width=.8\columnwidth]{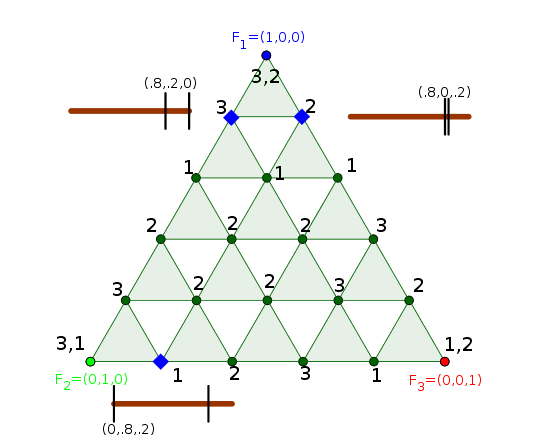}
	\end{center}
	\caption{\label{fig:single-agent-negative}
		Three points representing the same physical partition.
		}
\end{figure}

\section{Cutting Mixed Cakes}
\label{sec:mixed}
\subsection{The Consistency Condition}
\label{sub:permutation}
The first step in handling a mixed cake is to find boundary conditions that are satisfied for \emph{all} agent labelings,  regardless of whether the cake is good, bad or mixed. 
Our boundary condition is based on the observation
that \emph{different points on the boundary of the partition-simplex may represent the same physical cake-partition}. For example, consider the three diamond-shaped points in Figure \ref{fig:single-agent-negative}. In each of these points, the set of pieces is the same: $\{[0,.8],[.8,1],\emptyset\}$. 
Therefore, a consistent agent will select the same piece in all three partitions, even though this piece might have a different index in each point. This means that the agent's label in each of these points uniquely determines the agent's labels in the other two points. For example, if the agent labels the top-left diamond point by ``3'', this means that he prefers the empty piece, so he must label the top-right diamond point by ``2'' and the bottom-left diamond point by ``1'' (as in the figure).

To formalize this boundary condition we need several definitions.

\newcommand{\perm}[1]{\pi_{-#1}}
\newcommand{\friend}[1]{f_{#1}}

\begin{definition}
\label{def:friends}
Two points in $\Delta^{n-1}$ are called \emph{friends} if they have the same ordered sequence of nonzero coordinates.
\end{definition}
For example, on $\Delta^{3-1}$, the points $(0,.2,.8)\in F_{-1}$ and $(.2,0,.8)\in F_{-2}$ and $(.2,.8,0)\in F_{-3}$ are friends, since their ordered sequence of nonzero coordinates is $(.2,.8)$. 
\ifdefined\FULLVERSION
But the point $(0,.8,.2)$ is not their friend since its ordered sequence of nonzero coordinates is $(.8,.2)$.
Note that if $x$ is in the interior of $\Delta^{n-1}$, then all its coordinates are nonzero, so it has no friends except itself. 
\fi

Since our boundary conditions have a bite only for friends, we will consider from now on only triangulations that are ``friendly'':
\begin{definition}
A triangulation $T$ is called \emph{friendly} if, for every vertex $x\in \vrt(T)$, all the friends of $x$ are in $\vrt(T)$.
\end{definition}

Our boundary condition is that the label of a vertex in $F_{-1}$ uniquely determines the labels of all its friends on the other faces. Specifically, 
consider a vertex $x_k\in F_{-k}$.
By definition of $F_{-k}$, the $k$-th coordinate of $x_k$ iz zero. 
If we move the $k$-th coordinate of $x_k$ to position 1 and push coordinates $1,\ldots,k-1$ one position rightwards, we get a vertex on $F_{-1}$ that is a friend of $x_k$; denote it by $\friend{k}(x_k)$. Since the triangulation is friendly, it contains $\friend{k}(x_k)$.

Suppose that the label of $\friend{k}(x_k)$ is $l$. Then the label on $x_k$ is:
\begin{align}
\label{eq:perm}
\perm{k}(l) := 
\begin{cases} 
	k&      l=1 \text{~~~~~~[agent prefers  empty piece]}\\
	l-1& 	1< l \leq k \\
	l & 	l>k 
\end{cases}
\end{align}
For every $k$, the function $\perm{k}$ is a permutation (a bijection from $[n]$ to $[n]$). $\perm{1}$ is the identity permutation. 
\ifdefined\FULLVERSIONPERMTABLE
Table \ref{tab:perm} shows the three permutations for $n=3$: $\perm{1}$, $\perm{2}$, $\perm{3}$.
\begin{table}
\begin{center}
\begin{tabular}{|c|c|c|c|c|}
	\hline 
	Preferred piece: &  Empty & Left & Right &\{ER\}ELRE\{EL\} \\ 
	\hline 
	Label on $F_{-1}$:& 1 & 2 & 3 & \{13\}1231\{12\} \\ 
	\hline 
	Label on $F_{-2}$:& 2 & 1 & 3 & \{23\}2132\{21\}\\ 
	\hline 
	Label on $F_{-3}$: & 3 & 1 & 2 & \{32\}3123\{31\} \\ 
	\hline 
\end{tabular}
\end{center}
\caption{\label{tab:perm} Label-permutations 
that satisfy Definition \ref{def:permutation} for $n=3$.
The rightmost column is provided as an example. It corresponds to the labeling in each edge in Figure \ref{fig:single-agent-negative}. Note that the labeling always goes from the vertex with the lower index (the Left vertex) to the vertex with the higher index (the Right vertex). E means that the agent prefers the Empty piece, R means the Right piece and L means the Left piece. Braces imply that there are multiple labels on the same point.
}
\end{table}
\else
The table below shows the three permutations for $n=3$: $\perm{1}$, $\perm{2}$ and $\perm{3}$. The rightmost column is an illustration corresponding to the sequence of labels on each face in Figure \ref{fig:single-agent-negative}:
\\
\begin{center}
\begin{tabular}{|c|c|c|c|c|}
	\hline 
	Preferred piece: &  Empty & Left & Right &\{ER\}ELRE\{EL\} \\ 
	\hline 
	Label on $F_{-1}$:& 1 & 2 & 3 & \{13\}1231\{12\} \\ 
	\hline 
	Label on $F_{-2}$:& 2 & 1 & 3 & \{23\}2132\{21\}\\ 
	\hline 
	Label on $F_{-3}$: & 3 & 1 & 2 & \{32\}3123\{31\} \\ 
	\hline 
\end{tabular}
\end{center}
\fi

\begin{definition}
\label{def:permutation}
A labeling $L: \vrt(T)\to 2^{[n]}$ is \textbf{consistent} if, for every
$k\in[n]$ and vertex $x_k\in F_{-k}$:
\begin{align*}
L(x_k) = \perm{k}(L(\friend{k}(x_k)))
\end{align*}
where $\perm{k}$ is defined by (\ref{eq:perm}),
and $\friend{k}(x_k)$ is a friend of $x_k$ on $F_1$, derived from $x_k$ by moving its $k$-th coordinate to position 1.
\end{definition}
Note that $L(x_1)$ may be a set of more than one label, and in this case, consistency implies that $L(x_k)$ is a set with the same number of labels. For example, if $x_1\in F_{-1}$ and $L(x_1)=\{1,2\}$ and $x_3\in F_{-3}$ then $L(x_3)=\perm{3}(\{1,2\}) = \{\perm{3}(1),\perm{3}(2)\} = \{3,1\}$.

Figures
\ref{fig:single-agent},  \ref{fig:single-agent-negative} 
show examples of consistent labelings.

Consistency has implications on the possible sets of labels on faces $F_J$ where $|J|\leq n-2$. Such faces are intersections of two or more $n-1$-dimensional faces. 
For example, let $x$ be the main vertex $F_3=(0,0,1)$. Then, 
$x$ is a friend of itself, with 
$\friend{2}(x) = x$. Therefore, consistency implies that $L(x) = \perm{2}(L(x))$. Hence, $L(x)$ contains $2$ if-and-only-if it contains $1$.
This makes sense: since all empty pieces are identical, the agent prefers an empty piece if and only if it prefers all empty pieces.
This is generalized in the following%
\ifdefined\FULLVERSION
\else
lemma, whose proof appears in the full
 version
\fi
:
\begin{lemma}
\label{lem:zeros}
Let $L$ be a consistent labeling. Then, for every vertex $x\in F_{[n]\setminus J}$, either $L(x)\cap J = J$ or $L(x) \cap J = \emptyset$.
\end{lemma}
\ifdefined\FULLVERSION
\begin{proof}
We first prove the lemma for the special case where 
$J = [k]$ for some $k\in [n]$. I.e, the first $k$ coordinates of $x$ are 0.
Now, $x\in F_{-k}$, and $x = \friend{k}(x)$ (it is a friend of itself). 
Therefore, consistency implies that $L(x)=\perm{k}(L(x))$.
By looking at the function $\perm{k}$, it is evident that, if $L(x)$ contains any element of $[k]$, it must contain them all.

We now consider the general case, where $J = \{i_1,\ldots,i_k\}$ for some $k$ indices in $[n]$. Suppose $i_1 < \cdots < i_k$ and let:
\begin{align*}
y = \friend{i_k}(\friend{i_{k-1}}(\ldots \friend{i_1}(x)))
\end{align*}
By the consistency of $L$:
\begin{align}
\label{eq:Lv}
L(x) = \perm{i_1} (\perm{i_2} ( \cdots \perm{i_k}(L(y))))
\end{align}
Additionally, 
\begin{align}
\label{eq:Jv}
J = \perm{i_1} (\perm{i_2} ( \cdots \perm{i_k}([k])))
\end{align}
We already proved that the lemma holds 
for $y$, whose set of zero coordinates is $[k]$. 
Hence, by \eqref{eq:Lv} and \eqref{eq:Jv} it also holds for $x$, whose set of zero coordinates is $J$.
\end{proof}

Based on Lemma \ref{lem:zeros}, given a labeling $L$ and a vertex $x\in F_{[n]\setminus J}$, we say that:
\begin{itemize}
\item $x$ is a \emph{positive vertex} if $L(x)\cap J = \emptyset$;
\item $x$ is a \emph{negative vertex} if $L(x)\cap J = J$.
\end{itemize}
In a positive vertex the agent prefers a nonempty piece; in a negative vertex the agent prefers an empty piece.
\fi

Our goal now is to prove that, if all $n$ agent-labelings are consistent, 
then an envy-free simplex exists. We proceed in two steps.
\begin{itemize}
\item If all labelings $L_1,\ldots,L_n$ are consistent, then there exists a single consistent combined labeling $L^W$ (Subsection \ref{sub:combining}).
\item If a labeling is consistent, then it has a fully-labeled simplex (Subsections \ref{sub:degree}-\ref{sub:permutation-degree}).
\end{itemize}
\subsection{Combining n labelings to a single labeling}
\label{sub:combining}
The consistency condition is valid for a single agent. We have to find a way to combine $n$ different consistent labelings into a single consistent labeling. For this we need several definitions.
\begin{definition}
An \emph{ownership-assignment} of a triangulation $T$ is a function from the vertices of the triangulation to the set of $n$ agents, $W: \vrt(T)\to \{A_1,\ldots,A_n\}$.
\end{definition}
\begin{definition}
Given a triangulation $T$, $n$ labelings $L_1,\ldots,L_n$, and an ownership assignment $W$, the \emph{combined labeling} $L^W$ is the labeling that assigns to each vertex in $\vrt(T)$ the label/s assigned to it by its owner. I.e., if $W(x)=A_i$, then $L^W(x):= L_i(x)$.
\end{definition}

\begin{definition}\label{def:diverse}
An ownership-assignment $W$ is called:

(a) \emph{Diverse} --- if in each sub-simplex in $T$, each vertex of the sub-simplex has a different owner;

(b) \emph{Friendly} --- if it assigns friends to the same owner. I.e., for every pair $x,y$ of friends (see Definition \ref{def:friends}), $W(x)=W(y)$.
\end{definition}
The diversity condition was introduced by \citet{Su1999Rental}.
As an example, the ownership-assignment of Figure \ref{fig:triangulation-ownership} is diverse. However, it is not friendly. For example, the two vertices near $(1,0,0)$ are friends since their coordinates are $(.8,.2,0)$ and $(.8,0,.2)$, but they have different owners ($B,C$). This means that the combined labeling is not necessarily consistent.
\ifdefined\FULLVERSION
It is easy to construct a friendly ownership-assignment: go from $F_1$ towards $F_2$, assign the vertices to arbitrary owners, then assign the vertices from $F_1$ towards $F_3$ and from $F_2$ towards $F_3$ to the same owners. However, in general it will not be easy to extend this to a diverse assignment.

Does there always exist an ownership-assignment that is both diverse and friendly? The following lemma shows that the answer is yes.
\fi
Fortunately, there always exists an ownership-assignment that is both friendly and diverse.
\begin{figure}
	\begin{center}
		\includegraphics[width=.44\columnwidth]{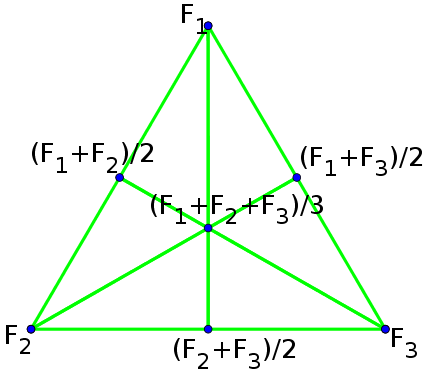}
		\hskip .04\columnwidth
		\includegraphics[width=.44\columnwidth]{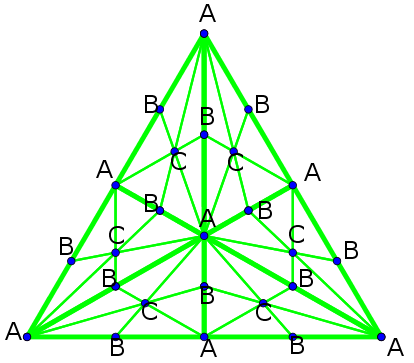}
	\end{center}
	\vskip -.4cm	
	\caption{\label{fig:ownership-bary}
	\textbf{Left}: barycentric subdivision of a triangle.
	\\
	\textbf{Right}: Barycentric triangulation of a triangle, with a friendly and diverse ownership assignment (here $A,B,C$ are agents $A_1,A_2,A_3$).
	}
	\vskip -.4cm
\end{figure}

\begin{lemma}
\label{lem:friendly-and-diverse}
For any $n\geq 3$ and $\delta>0$, there exists a friendly triangulation $T$ of $\Delta^{n-1}$ where the diameter of each sub-simplex is $\leq \delta$, and an ownership-assignment of $T$ that is friendly and diverse.
\end{lemma}
\begin{proof}
The construction is based on the \emph{barycentric subdivision}.\footnote{the explanation follows the Wikipedia page ``barycentric subdivision''.}
\def\BARYCENTRICGENERAL{}
\ifdefined\BARYCENTRICGENERAL
The barycentric subdivision of a simplex with main vertices $F_1,\ldots,F_n$ is constructed as follows.

Pick a permutation $\sigma$ of the main vertices. For every prefix of the permutation, $\sigma_1,\ldots,\sigma_m$ (for $m\range{1}{n}$), define $v_m$ as their barycenter (arithmetic mean): $v_m := (\sigma_1+\cdots+\sigma_m)/m$. We call $v_m$ a \emph{level-$m$ vertex}. The vertices $v_1,\ldots,v_n$ define a subsimplex.

Each permutation yields a different subsimplex, so all in all, the barycentric subdivision of an $(n-1)$-dimensional simplex contains $n!$ subsimplices. Note that each sub-simplex has exactly one vertex of each level $m\in[n]$.
\else
It is defined for any $n$, but for simplicity we present it for $n=3$. Consider a 2-dimensional simplex (= a triangle) with vertices $F_1 F_2 F_3$. There are 6 different permutations of these vertices. Each such permutation defines a sub-triangle whose vertices are: 1. the first vertex, e.g. $F_1$; 2. the average of the first two vertices, e.g. $(F_1+F_2)/2$; 3. the average of all three vertices, e.g. $(F_1+F_2+F_3)/3$. See Figure \ref{fig:ownership-bary}/Left.
The \emph{level} of a vertex in the subdivision is the number of vertices of the original triangle that are averaged in it (1, 2 or 3 respectively).
\fi

By recursively applying a barycentric subdivision to each subsimplex (as in Figure \ref{fig:ownership-bary}/Right), we get iterated \emph{barycentric triangulations}. 
The ownership assignment is determined by the levels of vertices in the last subdivision step: 
each vertex with level $i$ is assigned to agent $A_i$ (see Figure \ref{fig:ownership-bary}/Right). This ownership assignment is:
\begin{itemize}
\item diverse --- since for every $i$, each subsimplex has exactly one vertex of level $i$.
\item friendly --- since, by the symmetry of the barycentric subdivision, every two friend-vertices have the same level.\qedhere
\end{itemize}
\end{proof}

\begin{lemma}
\label{lem:friendly-ownership}
Let $L_1,\ldots,L_n$ be consistent labelings of a friendly triangulation $T$.
If $W$ is a friendly ownership-assignment, then the combined labeling $L^W$ is consistent.
\end{lemma}
\begin{proof}
Consistency restricts only the labels of friends. Since all friends are labeled by the same owner, and the labeling of each owner is consistent, the combined labeling is consistent too.
\end{proof}
Lemmas \ref{lem:friendly-and-diverse} and \ref{lem:friendly-ownership} reduce the problem of finding an envy-free simplex with $n$ labelings, to the problem of finding a fully-labeled simplex with a single labeling. 
This is our next task.

\subsection{The Degree Lemma}
\label{sub:degree}
We want to prove that any consistent labeling has a fully-labeled simplex. For this we develop a generalization of Sperner's lemma.

In this subsection we will consider single-valued labelings. To differentiate them from the multi-valued labelings denoted by $L:\vrt(T)\to2^{[n]}$, we will denote them by $\ell: \vrt(T)\to [n]$.

We will use the following claim that we call the \emph{Degree Lemma}:
\begin{quote}
\emph{Let $\ell: \vrt(T)\to[n]$ a labeling of a triangulation $T$. The \emph{interior degree} of $\ell$ equals its \emph{boundary degree}.}
\end{quote}


To explain this lemma we have to explain what are ``interior degree'' and ``boundary degree'' of a labeling.%
\footnote{
The Degree Lemma can be proved as a corollary of much more general theorems in algebraic topology. See Corollary 3 in \citet{meunier2008combinatorial} and Corollary 3.1 in \citet{Musin2014Around}. For simplicity and self-containment we present it here using stand-alone geometric arguments.
Some of the definitions follow \citet{Matveev2006Lectures}.
}

Throughout this subsection, $Q$ denotes a fixed $n-1$-dimensional simplex in $\mathbb{R}^{n-1}$ 
whose vertices are denoted by $Q_1,\ldots,Q_n$.
$Q'$ denotes a fixed face of $Q$ of co-dimension 1 (so $Q'$ is an $n-2$-dimensional simplex). Most illustrations are for the case $n=4$.

\subsubsection{\textbf{Interior degree}}
\label{sub:deg-map}
Let $P$ be an $n-1$ dimensional simplex in $\mathbb{R}^{n-1}$.  
Let $g: P\to Q$ be a mapping that maps each of the $n$ vertices of $P$ to a vertex of $Q$. 
By basic linear algebra, there is a unique way to extend $g$ to an affine transformation from $P$ to $Q$. Define $\deg(g)$ as the sign of the determinant of this transformation:
\begin{itemize}
\item $\deg(g)=+1$ means $g$ is onto $Q$ and can be implemented by translations, rotations and scalings (but no reflections);
\item $\deg(g)=-1$ means $g$ is onto $Q$ and can be implemented by translations, rotations, scalings and a single reflection;
\item $\deg(g)=0$ means $g$ is not onto $Q$ (i.e., it maps the entire $P$ into a single face of $Q$ with dimension $n-2$ or less).
\end{itemize}
Every labeling $\ell: \vrt(P)\to [n]$ defines a mapping $g_\ell$ where for each vertex $v\in \vrt(P)$ whose label is $j$, we let $g_\ell(v)=Q_j$.
The pictures below show three such mappings with different degrees from different source simplices in $\mathbb{R}^3$ to the same target $Q$:
\footnote{
\label{fn:visualise}
To visualize the degree, imagine that you transform the source simplex until it overlaps the target simplex $Q$, such that each vertex labeled with $j$ overlaps $Q_j$. If you manage to do that without reflections then the degree is $+1$; otherwise it is $-1$. 
}
\begin{center}
\includegraphics[width=1\columnwidth]{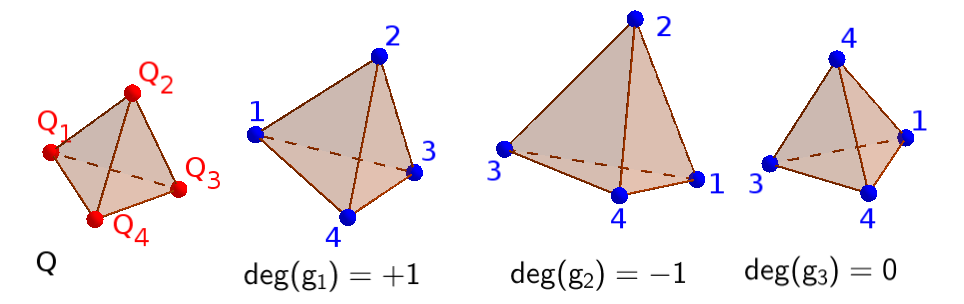}
\end{center}
We make several observations that relate the labeling to the degree.
\begin{enumerate}
\item If $P$ is \emph{fully-labeled} (each vertex has a unique label), then $g_\ell$ is onto $Q$, so $\deg(g_\ell)$ is either $+1$ or $-1$ (examples $g_1$ and $g_2$ above). %
If $P$ is not fully-labeled (two or more vertices have the same label), then $g_\ell$ is not onto $Q$ so $\deg(g_\ell)=0$ (ex. $g_3$).
\item Swapping two labels on $P$ corresponds to a reflection. Therefore, an odd permutation of the labels inverts the sign of $\deg(g_\ell)$, while an even permutation keeps $\deg(g_\ell)$ unchanged.
\end{enumerate}
\ifdefined\FULLVERSION
The following multiplicative property of the degree operator follows directly from the properties of determinants (or affine mappings). For every $g,g$: $\deg(g\circ g) = \deg(g)\cdot \deg(g)$
\fi

Let $T$ be a triangulation of some simplex and $\ell: \vrt(T)\to [n]$ a labeling. 
In each $n-1$ dimensional sub-simplex $t$ of the triangulation $T$, the labeling $\ell$ defines an affine transformation $g_{\ell,t}: t\to Q$.
The \emph{interior degree} of $\ell$ is defined as the sum of the degrees of all these transformations:
\begin{align*}
\ideg(\ell) &:= \sum_{t\in T} \deg(g_{\ell,t})
\end{align*}
Note that each fully-labeled sub-simplex of $T$ contributes either $+1$ or $-1$ to this sum and each non-fully-labeled sub-simplex contributes $0$. So if $\ideg(\ell)\neq 0$, there is at least 1 fully-labeled simplex.

\subsubsection{\textbf{Boundary degree}}
Consider now an $n-2$-dimensional simplex in $\mathbb{R}^{n-1}$. It is contained in a hyperplane and this hyperplane divides $\mathbb{R}^{n-1}$ into two half-spaces. 
Define an \emph{oriented simplex} in $\mathbb{R}^{n-1}$ as a pair of an $n-2$-dimensional simplex and one of its two half-spaces (so each such simplex has two possible orientations).

Let $P',Q'$ be two oriented simplices in $\mathbb{R}^{n-1}$.
Let $g$ be a mapping that maps each vertex of $P'$ to a vertex of $Q'$, and maps the half-space attached to $P'$ to the half-space attached to $Q'$. There are infinitely many ways to extend $g$ to an affine transformation, but all of them have the same degree. Three examples are shown below; an arrow denotes the half-spaces attached to the simplex:${}^{\ref{fn:visualise}}$
\begin{center}
\includegraphics[width=1\columnwidth]{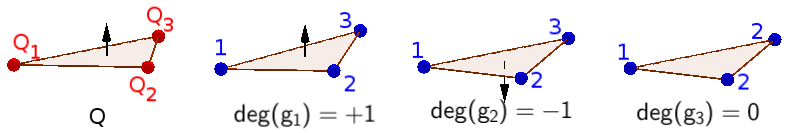}
\end{center}
Consider now an $n-2$-dimensional simplex $P'$ that is a face of an $n-1$-dimensional simplex $P$. Since $P$ is convex, it is entirely contained in one of the two half-spaces adjacent to $P'$. We orient $P'$ by attaching to it the half-space that contains $P$  (figuratively, we attach to $P'$ an arrow pointing inwards, towards the interior of $P$).

Let $Q'$ be a fixed $n-2$-dimensional face of $Q$ oriented towards the interior of $Q$. Let $\ell:\vrt(P')\to [n]$ be a labeling. If every label on $P'$ is one of the $n-1$ labels on $Q'$, then $\ell$ defines a mapping $g_\ell:P'\to Q'$ where for each vertex $v\in \vrt(P')$ whose label is $j$, we let $g_\ell(v)=Q_j$, and the half-space attached to $P'$ is mapped to the half-space attached to $Q'$. The same observations (1) and (2) above relate the labeling $\ell$ with the degree $\deg(g_\ell, Q')$. If some label on $P'$ is not one of the labels on $Q'$, then we define $\deg(g_\ell, Q')=0$.

It is convenient to define the degree of $g_\ell$ w.r.t. at all $n$ faces of $Q$ simultaneously. We denote by $\deg(g_\ell)$ (without the extra parameter $Q'$) the arithmetic mean of $\deg(g_\ell, Q')$ over all $n$ faces of $Q$:
\begin{align}
\label{eq:average-degree}
\deg(g_\ell) := {1\over n}\sum_{Q'\text{~face of~}Q} \deg(g_\ell, Q')
\end{align}
In this notation, if $\ell$ puts $n-1$ distinct labels on some face, then $\deg(g_\ell) = \pm {1\over n}$ since exactly one term in the mean is $\pm 1$ and the rest are zero. Otherwise, $\deg(g_\ell) = 0$ since all terms are zero.

Let $T$ be a triangulation of some simplex $P$ and let $\ell: \vrt(T)\to [n]$ be a labeling of the vertices of $T$. 
Denote by $\partial T$ the collection of $n-2$-dimensional faces of $T$ on the boundary of $P$. 
In each such face $t' \in \partial T$,
the labeling $\ell$ defines $n$ affine transformation $g_{\ell,t'}: t'\to Q'$ 
and their average degree $\deg(g_{\ell,t'})$ can be calculated as in \eqref{eq:average-degree}. 
The \emph{boundary degree} of $\ell$ is defined as the sum:
\begin{align*}
\bdeg(\ell) := \sum_{t' \in \partial T} \deg(g_{\ell,t'})
\end{align*}
We now re-state the degree lemma:
\begin{lemma}[Degree Lemma]
\label{lem:degree-T}
For every triangulation $T$ of a simplex $P$ and
every labeling $\ell:\vrt(T)\to[n]$:
\begin{align*}
\ideg(\ell) = \bdeg(\ell)
\end{align*}
\end{lemma}
\begin{proof}
\textbf{Part 1.}
We first prove the lemma for the case when the triangulation $T$ is trivial --- contains only the single $n-1$ dimensional simplex $P$.
In this case, the sum $\ideg(\ell)$ contains a single term --- $\deg(g_\ell)$ --- which can be either $-1$ or $0$ or $1$. The sum $\bdeg(\ell)$ contains $n$ terms --- one for each face of $P$.
We consider several cases depending on the number of distinct labels on $\vrt(P)$.

If the number of distinct labels is $n$ (i.e., $P$ is fully-labeled), then $\ideg(\ell)$ is $+1$ or $-1$.
Each face of $P$ is labeled with $n-1$ distinct labels so its degree is $+{1\over n}$ or $-{1\over n}$.
The same affine mapping $g_\ell$ that maps $P$ to $Q$, also maps each face $P'$ to each face $Q'$ with corresponding labels. Therefore, all terms have the same sign, and we get either $+1 = \sum {+1\over n}$ or $-1 = \sum {-1\over n}$, both of which are true.

If the number of distinct labels is $n-2$ or less, then $P$ is not fully-labeled so $\ideg(\ell)=0$. No faces of $P$ are labeled with $n-1$ distinct labels,  so $\bdeg(\ell)=0$ too.

If the number of distinct labels is $n-1$, then $P$ is not fully-labeled so $\ideg(\ell)=0$. 
$P$ has exactly two faces with $n-1$ distinct labels; let's call them $P'_+$ and $P'_-$. For each $s\in\{+,-\}$ and for each face $Q'\subseteq Q$, let $g'_s$ be the mapping from $P'_s$ onto $Q'$. 
It can be proved that $\deg(g'_+)=-\deg(g'_-)$
\ifdefined\FULLVERSION
.\footnote{
To see this, suppose that, before mapping $P'_+$ and $P'_-$ onto $Q'$, we first map $P'_+$ onto $P'_-$, with no reflection. Let $g$ be this mapping. Then, the half-space attached to $g(P'_+)$ does not contain the interior of $P$ (figuratively, the arrow attached to $g(P'_+)$ points outwards). 
To align the orientations we must use reflection, so an orientation-preserving mapping $g':P'_+\to P'_-$ must have $\deg(g')=-1$.
Since $g'_+ = g'_-\circ g'$, we have $\deg(g'_+)=-\deg(g'_-)$.
}
\else
 (see full version). 
\fi
Therefore $\deg(g'_+)  + \deg(g'_-) = 0$ and so  $\bdeg(\ell) = 0$ too.
The latter case is illustrated below, where $Q'~=~Q_1 Q_3 Q_4$; the degree is $+1$ at the top 134 face and $-1$ at the bottom 134 face.
\begin{center}
\vskip -.25cm
\includegraphics[width=.6\columnwidth]{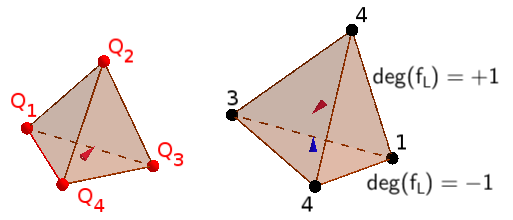}
\end{center}
\vskip -.2cm

\textbf{Part 2.}
We now prove the lemma for a general triangulation. 
For each $n-1$-dimensional sub-simplex $t\in T$, denote by $\ell_t$ the labeling $\ell$ in $t$,
and by $\ell_{t'}$ the labeling on its $n-2$-dimensional face $t'$. Then:
\begin{align*}
\ideg(\ell) &= \sum_{t\in T} \ideg(\ell_t)
&= \sum_{t\in T} \sum_{t' \text{~face of~}t} \bdeg(\ell_{t'}) && \text{By Part 1.}
\end{align*}
The sum in the right-hand side counts all $n-2$-dimensional faces in $T$ --- both on the boundary $\partial T$ and on the interior. Each face on the boundary is counted once since it belongs to a single sub-simplex, while each face in the interior is counted twice since it belongs to two sub-simplices.
The orientations of this face in its two sub-simplices are opposite, since the interiors of these sub-simplices are in opposite directions of the face. This is illustrated below:
\begin{center}
\vskip -.25cm
\includegraphics[width=.6\columnwidth]{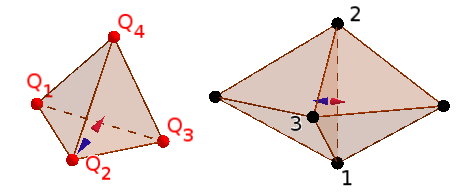}
\end{center}
\vskip -.3cm
Therefore, the two contributions of this face to $\bdeg(\ell_t)$ cancel out, 
and the right-hand side becomes
$\sum_{t'\in \partial T} \bdeg(\ell_{t'}) = \bdeg(\ell)$.
\end{proof}

An illustration of the Degree Lemma for $n=3$ is shown below:
\begin{center}
	\includegraphics[width=.8\columnwidth]{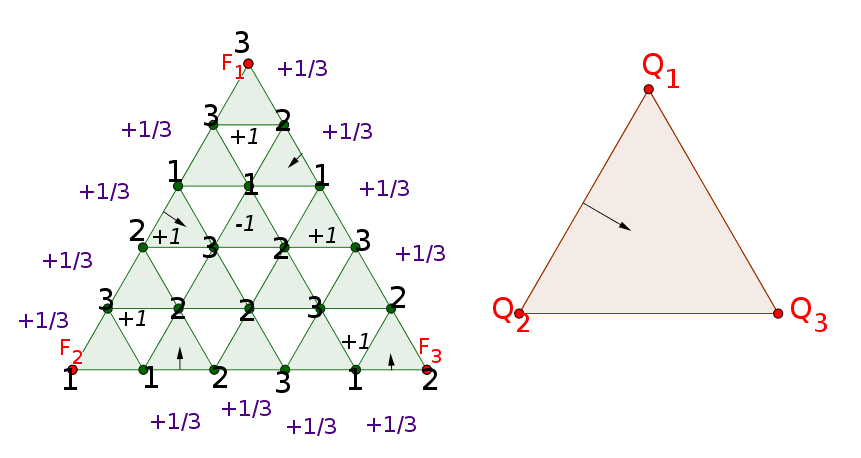}
	\label{fig:mapping-inner}
\end{center}
There are six fully-labeled triangles: in 
five of them, the transformation $g_\ell$ onto $Q$ requires no reflection so its degree is $+1$. In the sixth, the transformation $g_\ell$ onto $Q$ requires a single reflection so its degree is $-1$. Therefore: 
\hskip 1cm
$
\ideg(\ell)=+5-1=4
$.

At the boundary of $P$ there are four edges labeled $1,2$. The arrow adjacent to each edge indicates its orientation. Each of these edges can be transformed onto $Q'$ with no reflection while preserving the inwards orientation, so their degree is +1/3.
The same is true for the four edges labeled $2,3$ and $3,1$. 
 Therefore: $
\bdeg(\ell) = 12\cdot (1/3) = 4
$.


\subsection{Consistency $\scriptstyle\to$ Nonzero Boundary Degree}
\label{sub:permutation-degree}
The Degree Lemma reduces the problem of proving existence of a fully-labeled simplex, to the problem of proving that the boundary-degree is non-zero. Therefore, our next goal is to prove that every consistent labeling has a non-zero boundary degree. However, there is a technical difficulty: consistency is defined for multi-valued labelings, while the degree is only defined for single-valued labelings.
For the purpose of envy-free cake-cutting, we can convert a multi-valued labeling $L:\vrt(T)\to 2^{[n]}$ to a single-valued labeling
$\ell:\vrt(T)\to[n]$ by simply selecting, for each vertex $x\in \vrt(T)$, a single label from the set $L(x)$. In effect, we select for the agent one of his preferred pieces in that partition; this does not harm the envy-freeness.
If $\ell$ is created from $L$ using such a selection, we say that $\ell$ is \emph{induced by $L$}, and write $\ell\sim L$.

For our purposes, it is sufficient to prove that every consistent labeling $L$ induces \emph{at least one} labeling $\ell$ with non-zero boundary degree.
For this, it is sufficient to prove that the \emph{sum} of boundary degrees, taken over all labelings $\ell$ induced by $L$, is nonzero: $\sum_{\ell\sim L}{\bdeg(\ell)}\neq 0$.
The first step is the following lemma that relates the boundary degree to consistency.
\begin{lemma}
\label{lem:constistency-degree}
Let $L: \vrt(T)\to 2^{[n]}$ be a consistent labeling of a friendly triangulation of $\Delta^{n-1}$. 
Let $L[F_{-1}]$ be the restriction of $L$ to the face $F_{-1}$. Then:
\begin{align*}
\sum_{\ell\sim L}\bdeg(\ell) = n\cdot \sum_{\ell\sim L}\bdeg(\ell[F_{-1}])
\end{align*}
\end{lemma}
\begin{proof}
Rewrite the left-hand side as a sum over all $n-2$ dimensional sub-simplices $t$ on the boundary of $\Delta^{n-1}$,
and then rewrite this boundary as a union of the  faces $F_{-k}$ for $k\in[n]$:
\begin{align}
\label{eq:sum-bdeg}
\sum_{\ell\sim L}\bdeg(\ell) =
\sum_{t \in T(\partial \Delta^{n-1})} 
\sum_{\ell\sim L}
\bdeg(\ell[t])
=
\sum_{k=1}^n
\sum_{t_k \in T(F_{-k})} 
\sum_{\ell\sim L}
\bdeg(\ell[t_k])
\end{align}
For each subsimplex $t_k$ in $T(F_{-k})$, let $\friend{k}(t_k)$ be a subsimplex in $F_{-1}$, whose vertices are the friends of the vertices of $t_k$. Recall from Definition \ref{def:permutation} that $\friend{k}(t_k)$ is derived from $t_k$ by moving the zero coordinate  in each vertex of $t_k$ from position $k$ to position $1$ and pushing its other coordinates rightwards.
Since $T$ is a friendly triangulation, it contains $\friend{k}(t_k)$.
The function $\friend{k}$ is a linear mapping from $\Delta^{n-1}$ to itself: it is a permutation of coordinates, and the permutation is even iff $k-1$ is even. Therefore, the degree of this mapping is $+1$ if $k-1$ is even and $-1$ if it is odd.
An illustration is shown in Figure \ref{fig:consistency-orientation}.
\begin{figure}
\begin{center}
\includegraphics[width=.45\columnwidth]{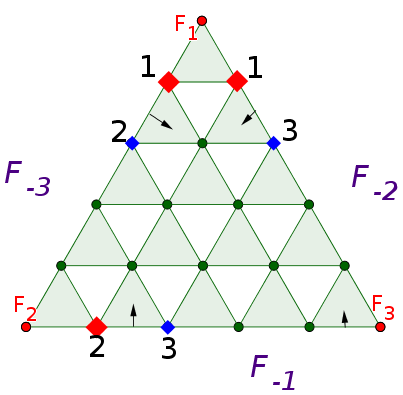}
\includegraphics[width=.45\columnwidth]{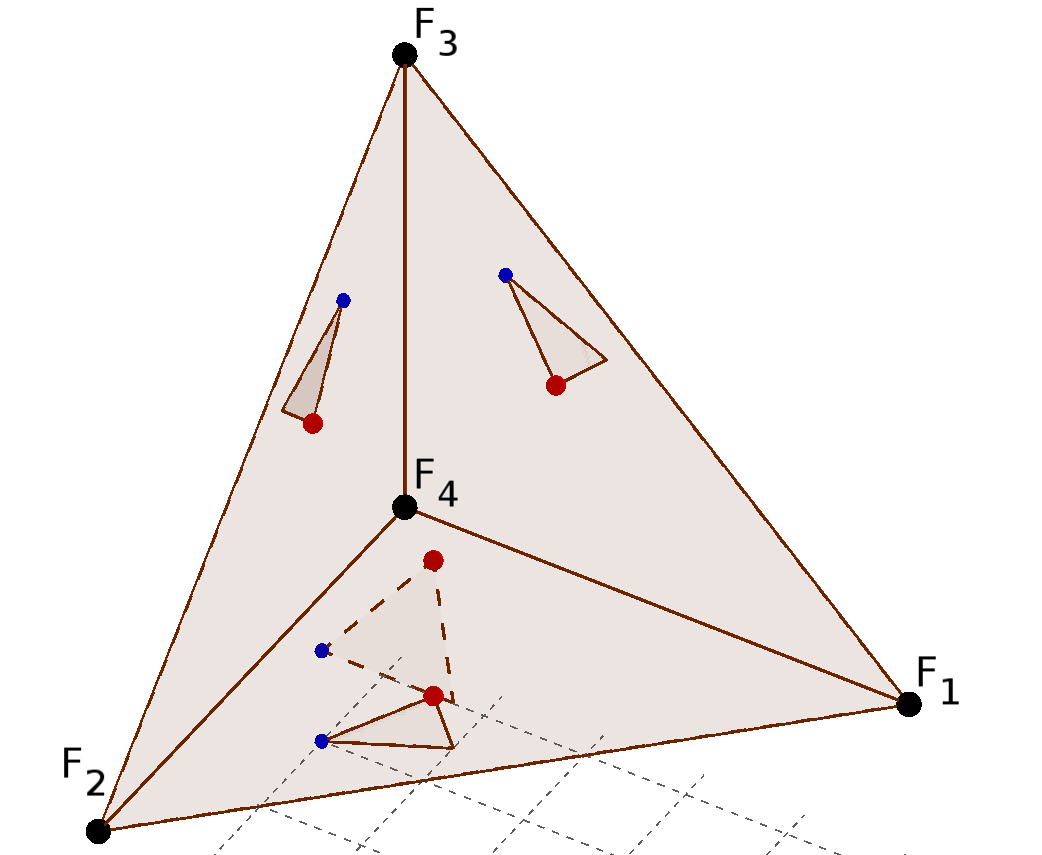}
\end{center}
\caption{
\label{fig:consistency-orientation}
An illustration of friend-simplices on the boundary of $\Delta^{n-1}$ for $n=3$ (left) and for $n=4$ (right).
The  $n$ vertices with the big-red marks are friends, and the $n$ vertices with the small-blue marks are friends.
The orientation-keeping transformation from a simplex in face $F_{-i}$ to its friend in face $F_{-k}$ can be done with no reflections if $i-k$ is even and with a single reflection if $i-k$ is odd.
}
\end{figure}

By consistency, for each vertex $x_k$ of $t_k$: $
L(x_k) = \perm{k}(L(\friend{k}(x_k)))
$.
The permutation $\perm{k}$, too, is even if-and-only-if $k-1$ is even. 
This is evident from the definition of $\perm{k}$ in \eqref{eq:perm}. The case $n=3$ is illustrated in Table \ref{tab:perm}: $\perm{1}$ is even (the identity permutation), $\perm{2}$ is odd (maps 123 to 213) and $\perm{3}$ is even (maps 123 to 312). 

Since $\friend{k}$ and $\perm{k}$ have the same sign, the contribution of $t_k$ to the sum of degrees is exactly the same as the contribution of $\friend{k}(t_k)$. 
Each subsimplex $t$ in $F_{-1}$ is a friend of exactly $n$ simplices $t_k$ in $F_{-k}$, for $k\in[n]$.
Therefore, \eqref{eq:sum-bdeg} equals:
\begin{align*}
\sum_{k=1}^n
\sum_{t \in T(F_{-1})} 
\sum_{\ell\sim L}
\bdeg(\ell[t])
=
n\cdot 
\sum_{\ell\sim L}
\bdeg(L[F_{-1}]).
\end{align*}
\end{proof}

Based on Lemma \ref{lem:constistency-degree},
we now prove that when $n=3$, the boundary degree of a consistent labeling is nonzero.

\begin{lemma}
	\label{lem:permutation-degree}
Let $L: \vrt(T)\to 2^{[3]}$ be a consistent labeling of a friendly triangulation of $\Delta^{3-1}$.
Then, $L$ induces a single-valued labeling $\ell:\vrt(T)\to [3]$ with:
\begin{align*}
\bdeg(\ell)
\not\equiv 0~\operatorname{mod}~3 
\end{align*}
\end{lemma}
\begin{figure}
\begin{center}
\includegraphics[width=0.44\columnwidth]{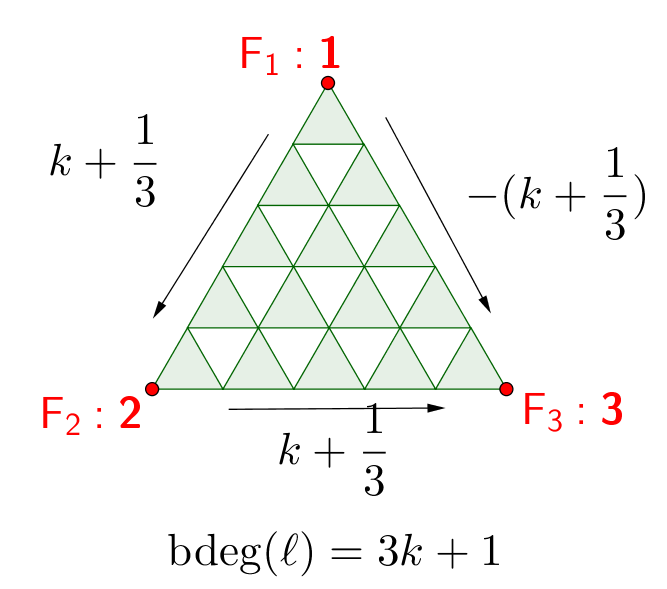}
\hskip 0.04\columnwidth
\includegraphics[width=0.44\columnwidth]{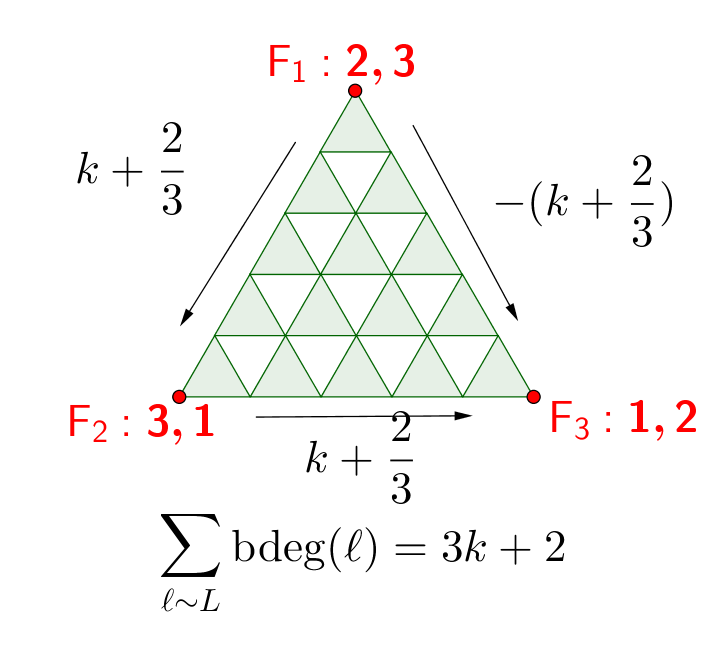}
\end{center}
\caption{\label{fig:bdeg-of-consistent-labelings}
\textbf{Left}: Boundary degrees of labelings in the positive case.
The sequence of labels from $F_{1}$ to $F_2$ is the same as the sequence from $F_{2}$ to $F_3$ up to an even permutation on the labels, so they contribute the same amount $k+1/3$ to the boundary-degree.
The sequence of labels from $F_1$ to $F_3$ is the same up to an odd permutation on the labels, but when we travel around $\Delta^{3-1}$
counter-clockwise, we see these labels in the opposite order, so this face contributes to the boundary degree $--(k+1/3) = k+1/3$.
 \\
\textbf{Right}: Boundary degrees of labelings in the negative case.
}

\end{figure}
\begin{proof}
\ifdefined\FULLVERSION
\else
\vskip -.45cm
\fi
First, we simplify $L$ by removing multiple labels while keeping $L$ consistent. This can be done arbitrarily for any interior vertex, since these vertices are not bound by consistency. 
For any boundary vertex $x$ that is not a main vertex, we remove labels consistently. For example, if a label $x_1\in F_{-1}$ is originally labeled by $\{2,3\}$ and we remove the 2, then by consistency its friend $x_3\in F_{-3}$ is originally labeled by $\{1,2\}$ and we remove the 1.

For the main vertices, Lemma \ref{lem:zeros} implies that there are exactly two cases regarding the labels on the main vertices.

\textbf{Positive case}  (Figure \ref{fig:bdeg-of-consistent-labelings}/Left): For each main vertex $F_j$,  $j\in L(F_j)$ (this corresponds to the owner of the main vertices valuing the entire cake as weakly-positive).
\ifdefined\FULLVERSION
\footnote{
At first glance, one might think that in the positive case each face $F_{ij}$ should be labeled only with $i$ and $j$, since there always exists a non-empty piece with a positive value, so this case could be handled by Sperner's lemma. While this is true for a single agent-labeling, it is \emph{not} true for a combined labeling: it is possible that the labels on the main vertices belong to Alice while the adjacent labels on the faces $F_{ij}$ belong to Bob (as in Figure \ref{fig:ownership-bary}), and Bob might think that the entire cake is negative.
}
\fi
We remove all other labels from $F_j$. The labeling remains consistent and it is now single-valued so we denote it by $\ell$.

Let $\ell[F_{-1}]$ be the labeling $\ell$ restricted to the face $F_{-1}$. 
Its boundary degree is determined by the sequence of labels from $F_2$ to $F_3$ --- it equals the number of cycles of labels in the order $1-2-3$ minus the number of cycles of labels in the opposite order $3-2-1$. Since the sequence of labels starts with 2 (on $F_2$) and ends with 3 (on $F_3$), the net number of cycles is fractional --- it is $k+1/3$ for some integer $k$. 
By Lemma \ref{lem:constistency-degree}, $\bdeg(\ell) = 3\cdot \bdeg(\ell[F_{-1}]) = 3 k+1$.

\textbf{Negative case}  (Figure \ref{fig:bdeg-of-consistent-labelings}/Right): For each main vertex $F_j$, $L(F_j) = [n]\setminus \{j\}$
(this corresponds to the owner of the main vertices valuing the entire cake as strictly negative).
Here, by Lemma \ref{lem:zeros}, there is no way to remove labels while keeping $L$ consistent.
So $L$ induces $2^3 = 8$ single-valued labelings.
\ifdefined\FULLVERSION
 We cannot just pick one of the eight arbitrarily, since for each single selection, the boundary degree might be zero in some cases, as the induced labelings need not be consistent.
two such cases are illustrated below:%
\footnote{
At first glance, one might think that, if the triangulation is sufficiently fine, we will not have such anomalous cases, since by continuity, the label in each vertex sufficiently close to $F_j$ should be one of the labels in $F_j$. As in the previous footnote, this is true for a single agent-labeling but false for a combined labeling.
}
\begin{center}
\includegraphics[width=.42\columnwidth]{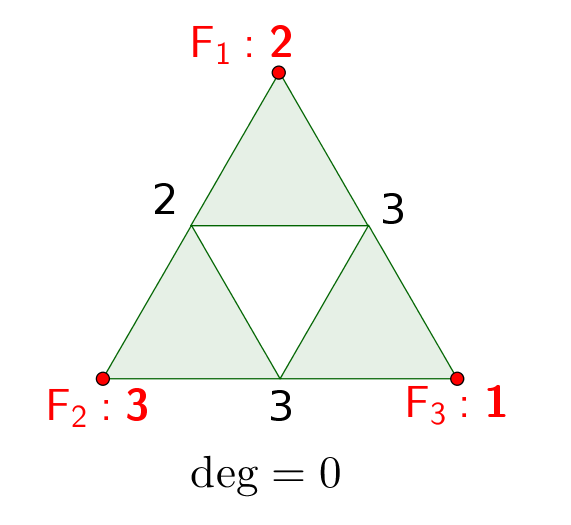}
\hskip 0.04\columnwidth
\includegraphics[width=.42\columnwidth]{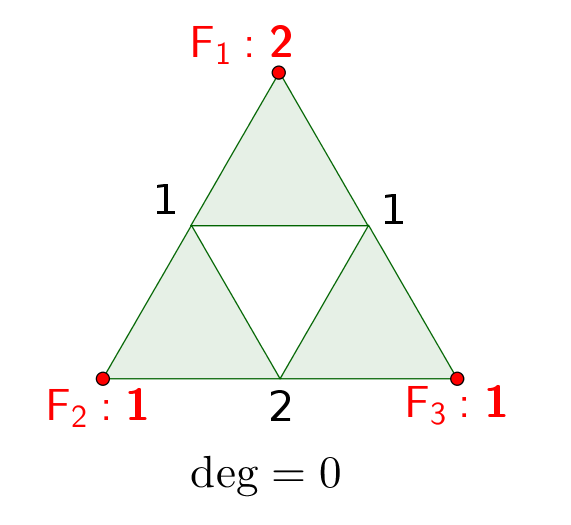}
\vskip -.5cm
\end{center}
\fi
However, the sum of boundary-degrees of these eight labelings is nonzero modulo 3.
Consider the face $F_{-1}$ 
and focus on the labels on the vertices $F_2,F_3$. If these labels are $3,2$ then the degree is $k_1-1/3$ for some integer $k_1$; if they are $3,1$ then it is $k_2+1/3$; if they are $1,2$ then it is $k_3+1/3$; if they are $1,1$ then it is $k_4$. We add all these numbers, then multiply by two for the two possible labels on $F_1$. We get $k+2/3$ for some integer $k$.
By Lemma \ref{lem:constistency-degree}, $\sum_{\ell\sim L} \bdeg(\ell) = 3\cdot \sum_{\ell\sim L} \bdeg(\ell[F_{-1}]) = 3 k+2$.

In all cases, $\sum_{l\sim L} \bdeg(\ell)
\not\equiv 0~\operatorname{mod}~3$. Hence, there is at least one $\ell\sim L$ with $\bdeg(\ell)
\not\equiv 0~\operatorname{mod}~3$, as claimed.
\end{proof}
\subsection{Tying the knots}
The final theorem of this section ties the knots.
\begin{theorem}
\vskip -.2cm
\label{thm:main}
For $n=3$ selective agents, there always exists a connected envy-free division of a mixed cake.
\end{theorem}
\begin{proof}
\vskip -.2cm
Let $T$ be a barycentric triangulation of the partition-simplex $\Delta^{n-1}$.
Let $W$ be a friendly and diverse ownership-assignment on $T$, which exists by 
\S\ref{sub:combining}.
Ask each agent to label the vertices he owns by the indices of his preferred pieces.

All $n$ labelings $L_1,\ldots,L_n$ are consistent (\S\ref{sub:permutation}).
Since $W$ is friendly, 
$L^W$ is consistent too
(\S\ref{sub:combining})%
. 
Therefore, there exists 
a single-valued labeling,
$\ell^W \sim L^W$, having 
$\bdeg(\ell^W)\neq 0$
(\S\ref{sub:permutation-degree}).
By the Degree Lemma (\S\ref{sub:degree}),
$\ideg(\ell^W)\neq 0$ too. Therefore $\ell^W$ has at least one fully-labeled simplex.
Since $W$ is diverse, a fully-labeled simplex of $\ell^W$ is an envy-free simplex.

All of the above can be done for finer and finer barycentric triangulations. This yields an infinite sequence of envy-free simplices. This sequence has a convergent subsequence. By continuity of  preferences, the limit of this subsequence is an envy-free division.
\end{proof}

We could not extend Lemma \ref{lem:permutation-degree} to $n>3$; it is left as a conjecture.
\begin{conjecture}
\label{cnj:perm-deg}
Let $L: \vrt(T)\to 2^{[n]}$ be a consistent labeling of a friendly triangulation of $\Delta^{n-1}$.
Then, $L$ induces a single-valued labeling $\ell:\vrt(T)\to [n]$ with:
\begin{align*}
\bdeg(\ell)
\not\equiv 0~\operatorname{mod}~n
\end{align*}
\end{conjecture}
If this conjecture is true, then Theorem \ref{thm:main} is true for any $n$.
\ifdefined\AAMAS
\footnote{
The difficulty comes from the need to handle multiple labels on the same vertex. 
With $n=3$ this is needed only in the three main vertices, where an agent who prefers an empty piece puts two labels. In Lemma \ref{lem:permutation-degree} Negative case, we explain why we cannot just pick 
one option arbitrarily, and handle this by summing over all $2^3$ options.
However, with $n\geq 4$, many more vertices can have multiple labels.
For example, on the one-dimensional face that connects $F_1$ and $F_2$,
every piece $k\geq 3$ is empty, so in \emph{any} vertex on that line, 
an agent who prefers an empty piece puts multiple labels.
The number of multi-label vertices is not fixed, so our proof technique does not work.
}
\fi
Section \ref{sec:special-case}
provides some evidence for the correctness of the conjecture by proving it in two special cases: (a) when Sperner's boundary condition holds for all faces with $n-2$ vertices or less, (b) when $n$ is a prime number.

\section{Finding an Envy-Free Division}
\label{sec:finding}
\citet{Stromquist2008Envyfree} proved that connected envy-free allocations cannot be found in a finite number of queries even when all valuations are positive, so the best we can hope for is an approximation algorithm.

The following simple binary-search algorithm can be used to find a fully-labeled sub-simplex in a labeled triangulation. It is adapted from \citet{Deng2012Algorithmic}:
\begin{enumerate}
\item If the triangulation is trivial (contains one sub-simplex), stop.
\item Divide the simplex into two halves, respecting the triangulation lines.
Calculate the boundary degree in each half.
\item Select one half in which the boundary degree is non-zero; perform the search recursively in this half.
\end{enumerate}
While \citet{Deng2012Algorithmic} present this algorithm for the positive case, it works whenever the boundary degree of the original simplex is non-zero. Then, in step 3, the boundary degree of at least one of the two halves is non-zero, so the algorithm goes on until it terminates with a fully-labeled simplex. This is the case when there are $n=3$ agents with arbitrary mixed valuations (Lemma \ref{lem:permutation-degree}).
If Conjecture \ref{cnj:perm-deg} is true, then this is also the case for any $n$.

To calculate the runtime of the binary search algorithm, 
suppose the triangulation is such that each side of the original simplex is divided into $D$ intervals. Then, the runtime complexity of finding a fully-labeled simplex is $O(D^{n-2})$ \citep{Deng2011Discrete}.

To calculate the complexity of finding a $\delta$-approximate envy-free allocation, we have to relate $D$ to $\delta$. 
In each barycentric subdivision, the diameter of the subsimplices is at most $n/(n+1)$ the diameter of the original simplex \citep{Munkres1996Elements}.
Hence, to get a barycentric triangulation in which the diameter of each sub-simplex is at most $\delta$, it is sufficient to perform
\ifdefined\FULLVERSION
$k$ steps of barycentric subdivision, where $k$ satisfies the inequality:
\begin{align*}
\bigg({n\over n+1}\bigg)^k &\leq \delta
\\
k &\geq {\ln \delta\over \ln{n\over n+1}}
  =   {\ln (1/\delta)\over \ln{(1+1/n)}}
 \\
 &\approx \ln(1/\delta)/(1/n) = n\ln(1/\delta).
\end{align*}
\else
$k \approx n\ln(1/\delta)$ steps of barycentric subdivision.
\fi
In each step, the number of intervals in each side is doubled, so $D \in \Theta (2^k) = \Theta({1/\delta}^n)$.
So the total runtime complexity of finding a $\delta$-approximate envy-free allocation using the barycentric triangulation is $O(1/\delta^{n(n-2)})$.%
\ifdefined\FULLVERSION
\footnote{
Note that $\delta$ is an additive approximation factor to the \emph{location} of the borderlines.
\cite{Deng2012Algorithmic} also consider an additive approximation to the \emph{values} (e.g., each agent values another agent's piece at most $\epsilon$ more than his own piece). To relate the $\delta$-factor to the $\epsilon$-factor, they add an assumption that the valuation functions are Lipschitz continuous with a constant factor.
}
\fi

\citet{Deng2012Algorithmic} note the slow convergence of the barycentric triangulation, and propose to use the \emph{Kuhn triangulation} instead. This triangulation looks similar to the equilateral triangulations shown in Figure \ref{fig:triangulation-ownership}. In this triangulation, $D = 1/\delta$ so the runtime complexity of the binary search is $O(1/\delta^{n-2})$. They prove that this is the best possible for selective agents. However, their triangulation does not support a diverse and friendly ownership-assignment. 

For $n=3$, we found a variant of the equilateral triangulation that does support a diverse and 
friendly ownership-assignment. The first two steps of this triangulation are illustrated below:
\begin{center}
\includegraphics[width=.4\columnwidth]{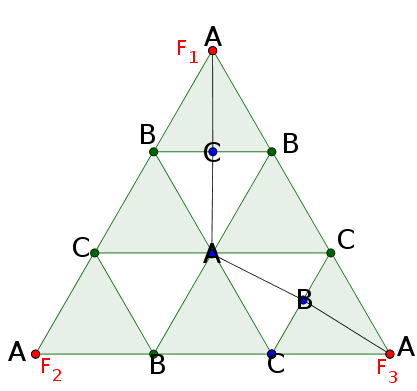}
\hskip .06\columnwidth
\includegraphics[width=.4\columnwidth]{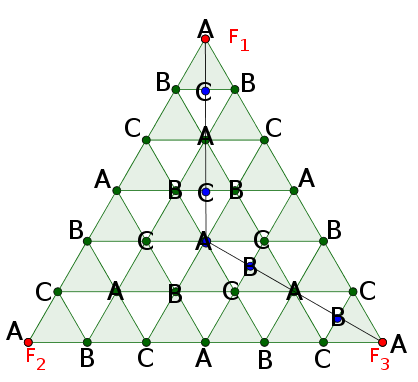}
\end{center}
So for $n=3$, a $\delta$-envy-free division can be found in time $O(1/\delta)$. 
Generalizing this ``trick'' to $n>3$ is left for future work.

\ifdefined\FULLVERSION
\section{Towards a General Solution}
\label{sec:special-case}
\newcommand{\lreplacen}[1]{\ell_{n\to #1}}
\newcommand{\Lreplacen}[1]{L_{n\to #1}}
This section presents preliminary results proving Conjecture \ref{cnj:perm-deg} in some special cases:
\begin{itemize}
\item When Sperner's boundary conditions are satisfied for all faces with at most $n-2$ vertices;
\item When $n$ is prime.
\end{itemize}
The latter result was previously proved by
\citet{meunier2018envy}.
The proof of Lemma \ref{lem:prime} below is inspired by their proof but uses more elementary arguments.

For both cases, we need the following lemma that relates Sperner's boundary condition to the boundary degree. It was proved before using different arguments, e.g. by \citet{frick2017achieving}

\begin{lemma}
\label{lem:sperner1}
Let $\ell:\vrt(T)\to[n]$ be a labeling of a triangulation of  $\Delta^{n-1}$.

Suppose that, for every $J\subseteq [n]$, all vertices on the face $F_J$ are only labeled with labels from $J$ (this is Sperner's boundary condition).

Then the boundary degree of $\ell$, w.r.t. a reference simplex $Q$ with the same orientation as $\Delta^{n-1}$, is one:
\begin{align*}
\bdeg(\ell) = 1
\end{align*}
\end{lemma}
Note that, by the Degree Lemma, this implies 
$\ideg(\ell) = 1$, which implies 
that the number of fully-labeled simplices must be odd. So Sperner's lemma can be seen as a special case of the Degree lemma.
\begin{proof}[Proof of Lemma \ref{lem:sperner1}]
The proof is by induction on $n$.

\textbf{Base: $n=3$}. By Sperner's condition,
the face $F_{12}$ is  labeled with only $1$ and $2$. One end of this face is labeled $1$ and the other end $2$, so the number of $12$ edges is one plus the number of $21$ edges. Therefore this face contributes $1/3$ to the boundary degree.
The same considerations are true for the faces $F_{23},F_{31}$. Therefore, $\bdeg(\ell)=3\cdot 1/3 = 1$.

\textbf{Step:} We assume the lemma is true for $n-1$ and prove it for $n$. 
By Sperner's condition, the face $F_{-n}$ is labeled only by labels from $[n-1]$. Let $\ell':F_{-n}\to [n-1]$ be the restriction of $\ell$ to the face $F_{-n}$.
Now, $F_{-n}$ is an $n-2$-dimensional simplex labeled by a labeling $\ell'$ that satisfies Sperner's boundary condition. 
By the induction assumption, the boundary degree of $\ell'$ is $1$. Therefore, by the Degree Lemma, $\ideg(\ell')=1$. 
So 
the net number of sub-simplices labeled with $n-1$ distinct labels from $[n-1]$ on $F_{-n}$ (positive minus negative) is 1. Therefore, $F_{-n}$ contributes $1/n$ to $\bdeg(\ell)$.
Exactly the same considerations, with only renaming of the labels, are true for all $n$ faces of $\Delta^{n-1}$. Hence, $\bdeg(\ell) = n\cdot 1/n = 1$.
\end{proof}

\newcommand{\multiset}[1]{\mathbb{N}^{[#1]}}

For the following results, we need a way to convert a labeling
with labels from $[n]$ to a labeling with labels only from $[n-1]$.
For this, we need to generalize the definition of a (multi-valued) labeling so that it can return a multi-set of labels rather than just a set. We denote the collection of multi-sets of labels from $[n]$ by $\multiset{n}$.

Let $L:\vrt(T)\to \multiset{n}$ be a labeling that assigns to each vertex of $T$ a multi-set of labels from $[n]$.
For every $j\in n$, we define $\Lreplacen{j}$ as a labeling obtained from $L$ by replacing each occurrence of $n$ by $j$ (we use multi-sets to ensure that labels are not merged during the replacement).
Obviously $\Lreplacen{n}\equiv L$, and for every $j<n$, $\Lreplacen{j}$ is into $\multiset{n-1}$.
We define $\lreplacen{j}$ analogously for single-valued labelings.

The following lemma reduces a labeling into $\multiset{n}$, to a function of labelings into $\multiset{n-1}$:

\begin{lemma}
\label{lem:sumreplace}
For every labeling $L:\vrt(T)\to\multiset{n}$, 
even when restricted to a part of $\Delta^{n-1}$ (e.g. a single face):
\begin{align*}
\sum_{\ell\sim L}
\bdeg(\ell) 
\equiv
 - \sum_{j=1}^{n-1} 
\sum_{\ell\sim  L}
\bdeg(\lreplacen{j})
\mod{1}
\end{align*}
(``$\equiv \mod{1}$'' means that the difference is a whole number).
\end{lemma}
\begin{proof}
We first prove the lemma for a single-valued labeling $\ell$. 
Denote by $\#_{-i}(\ell)$, the net number of sub-simplices on the boundary of $\Delta^{n-1}$, that are labeled with the $n-1$ distinct labels $[n]\setminus \{i\}$ (``net'' means positively-oriented minus negatively-oriented). Each such sub-simplex contributes $1/n$ to $\bdeg(\ell)$, so 
$
\bdeg(\ell) 
= {1\over n}\sum_{i\in [n]} 
\#_{-i}(\ell)
$.

We separate this sum to two terms:
$
{1\over n}\cdot \#_{-n}(\ell)
+
{1\over n}\sum_{i\in [n-1]} 
\#_{-i}(\ell)
$.
We sum each term over all $n$ replacements of $\ell$:
\begin{align*}
\sum_{j=1}^n \bdeg(\lreplacen{j})
=
{1\over n}\cdot \sum_{j=1}^n \#_{-n}(\lreplacen{j})
+
{1\over n} \sum_{j=1}^n  \sum_{i\in [n-1]} 
\#_{-i}(\lreplacen{j})
\end{align*}
Consider the two terms in the right-hand side. 

The term containing $\#_{-n}$ counts sub-simplices labeled with labels in $[n-1]$. Such sub-simplices are not affected by replacing the label $n$. Therefore $\#_{-n}(\lreplacen{j})$ is independent of $j$, so the leftmost term in the sum is ${1\over n}\cdot n\cdot \#_{-n}(\ell) = \#_{-n}(\ell)$. This is a whole number.

The term containing $\#_{-i}$ 
counts sub-simplices with labels in $[n]\setminus \{i\}$, where $n$ is replaced by $j$. 
Such sub-simplices have non-zero contribution 
only when $j=n$ (the label $n$ is not replaced) and when $j=i$ (the label $n$ is replaced with the missing label $i$), since these are the only cases where the sub-simplex remains with $n-1$ distinct labels. In these two cases, the sign of the sub-simplex is opposite. Therefore, the entire rightmost term in the sum is zero. We conclude that  $\sum_{j=1}^n \bdeg(\lreplacen{j})=\#_{-n}(\ell)$ --- a whole number. 

We separate this sum to $\bdeg(\ell) + 
\sum_{j=1}^{n-1} \bdeg(\lreplacen{j})$.
This sum is a whole number, so it equals 0  modulu 1, so $\bdeg(\ell) = -\sum_{j=1}^{n-1} \bdeg(\lreplacen{j}) \mod{1}$.
The lemma now follows by just summing each side of this equality over all $\ell \sim L$.
\end{proof}

Now, we apply Lemma \ref{lem:sumreplace} for the face $F_{-n}$.
This is an $n-2$-dimensional face, and each of the $\Lreplacen{j}$ for $j<n$ is a labeling of this face with labels from $[n-1]$.
Each of these labelings has an inner degree, which is the net number of sub-simplices labeled with $[n-1]$. 
By the Degree Lemma, its inner degree equals its boundary degree --- calculated over the $n-1$ faces of $F_{-n}$.
We denote the latter boundary degree by the operator $\bdeg_{n-1}$, to emphasize 
that it is calculated on a boundary with $n-1$ faces (of dimension $n-3$), in contrast to $\bdeg$, which is calculated on a boundary with $n$ faces (of dimension $n-2$).

The following lemma relates the boundary-degree of a labeling on the entire $n$-vertex simplex, to its boundary-degree on a single $n-1$-vertex face.

\begin{figure*}
\begin{center}
\includegraphics[width=.23\textwidth]{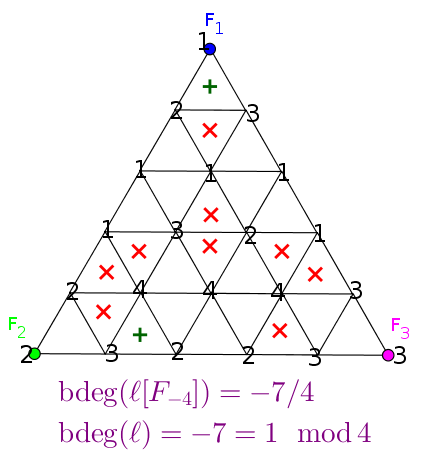}
\includegraphics[width=.21\textwidth]{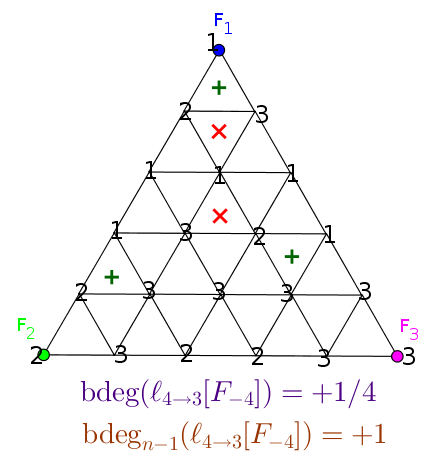}
\includegraphics[width=.21\textwidth]{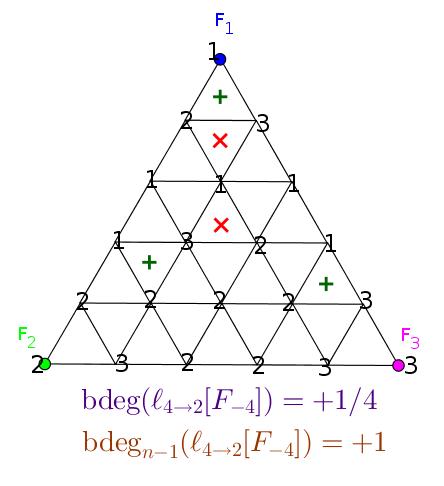}
\includegraphics[width=.21\textwidth]{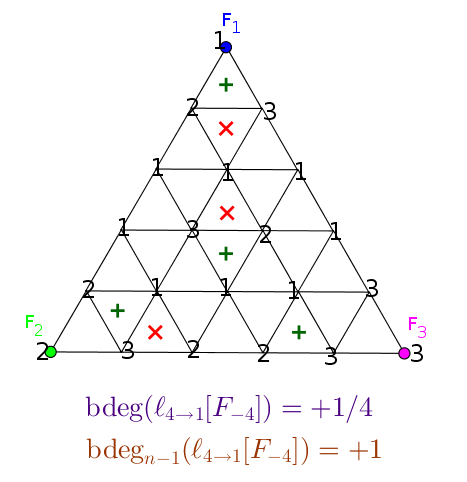}
\includegraphics[width=.08\textwidth]{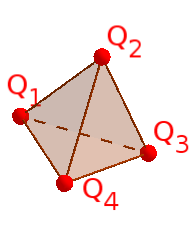}
\end{center}
\caption{
\label{fig:4agents-sperner}
\textbf{Left}: A consistent labeling $\ell$ on the face $F_{-4}$ of $\Delta^{4-1}$.
It satisfies the conditions of Lemma \ref{lem:sperner2} since it satisfies Sperner's boundary condition on the faces with at most $n-2$ vertices.
\\
\textbf{Middle}: 
The three labelings $\lreplacen{j}$ for $j\in [3]$.
Each of these labelings satisfies the conditions 
of Lemma \ref{lem:sperner1} since it satisfies Sperner's boundary condition.
\\
Each subsimplex with  ``+'' is positively  oriented and contributes $+1/4$ to the boundary degree; each subsimplex with ``X'' is negatively oriented and contributes $-1/4$. 
\\
\textbf{Right}:
The reference simplex $Q$.
}
\end{figure*}

\begin{lemma}
\label{lem:sumreplace2}
Let $L: \vrt(T)\to \multiset{n}$ be a consistent labeling of a friendly triangulation of $\Delta^{n-1}$.  Then:
\begin{align*}
\sum_{\ell\sim L}
\bdeg(\ell) 
\equiv
 - \sum_{j=1}^{n-1}
 \sum_{\ell\sim L} \bdeg_{n-1}(\lreplacen{j}[F_{-n}]) \mod{n}
\end{align*}
\end{lemma}
\begin{proof}
Lemma \ref{lem:constistency-degree} implies:
$\sum_{\ell\sim L}\bdeg(\ell) = n\cdot \sum_{\ell\sim L}\bdeg(\ell[F_{-n}])$.

Lemma \ref{lem:sumreplace} implies:
$\sum_{\ell\sim L}\bdeg(\ell[F_{-n}]) = 
- \sum_{j=1}^{n-1} \bdeg(\lreplacen{j}[F_{-n}]) \mod{1}$.

Therefore: $\sum_{\ell\sim L}\bdeg(\ell) = - n\cdot \sum_{j=1}^{n-1}
\sum_{\ell\sim L} \bdeg(\lreplacen{j}[F_{-n}]) \mod{n}$.

For every $j \in [n-1]$, 
each subsimplex with $n-1$ different labels in $\lreplacen{j}[F_{-n}]$ contributes $\pm 1/n$ to the boundary degree. 
However, for every $j \in [n-1]$, the only set of $n-1$ distinct labels in $\lreplacen{j}$ is $[n-1]$. Therefore, $\bdeg(\lreplacen{j}[F_{-n}])$ is $1/n$ times the net number of $[n-1]$-labeled subsimplices on $F_{-n}$.
Multiplying by $n$ gives
exactly the net number of 
$[n-1]$-labeled subsimplices, which is the $\ideg_{n-1}(\lreplacen{j}[F_{-n}])$.
By the Degree Lemma 
this equals
$\bdeg_{n-1}(\lreplacen{j}[F_{-n}])$.
\end{proof}

We now prove Conjecture \ref{cnj:perm-deg} for the special case in which Sperner's boundary condition is satisfied for all faces with at most $n-2$ vertices.

\begin{lemma}
\label{lem:sperner2}
Let $L: \vrt(T)\to 2^{[n]}$ be a consistent labeling of a friendly triangulation of $\Delta^{n-1}$.

Suppose that, for every $J\subseteq [n]$ with $|J|\leq n-2$, all vertices on the face $F_J$ are only labeled with labels from $J$.

Then, $L$ induces a single-valued labeling $\ell:\vrt(T)\to [n]$ with:
\begin{align*}
\bdeg(\ell)
\equiv 1~\mod~n.
\end{align*}
\end{lemma}
\begin{proof}
The lemma's assumption implies that we can simplify $L$ by removing multiple labels while keeping $L$ consistent (every vertex with two or more zero coordinates is labeled with an index of a non-zero coordinate, so it is not bound by consistency). Therefore we assume $L$ is single-valued and let $\ell = L$.

Moreover, the lemma's assumption implies that, on the face $F_{-n}$, each of the labelings $\lreplacen{j}$, for $j<n$, satisfies Sperner's boundary condition (in fact, $n$ does not appear on the boundary of $F_{-n}$, so all these labelings are identical on the boundary). Therefore, by Lemma \ref{lem:sperner1}, its boundary degree is 1:  $\bdeg_{n-1}(\lreplacen{j}[F_{-n}])=1$. 

Applying Lemma \ref{lem:sumreplace2} gives
$
\bdeg(\ell)
\equiv
- (n-1)\cdot 1 
\equiv
1
\mod n
$.
\end{proof}

Lemmas \ref{lem:sumreplace}, \ref{lem:sumreplace2} and \ref{lem:sperner2} are illustrated in Figure \ref{fig:4agents-sperner}.

\begin{remark*}
\emph{In the context of cake-cutting, Sperner's condition means that an agent always prefers a non-empty piece. 
The precondition of Lemma \ref{lem:sperner2} means that an agent prefers a non-empty piece whenever there are two or more empty pieces, but may prefer an empty piece when there is only one such piece. It is hard to relate this requirement to real-world agents. Therefore 
Lemma \ref{lem:sperner2} is interesting  theoretically more than practically.}
\end{remark*}

\begin{figure*}
\begin{center}
\includegraphics[width=.23\textwidth]{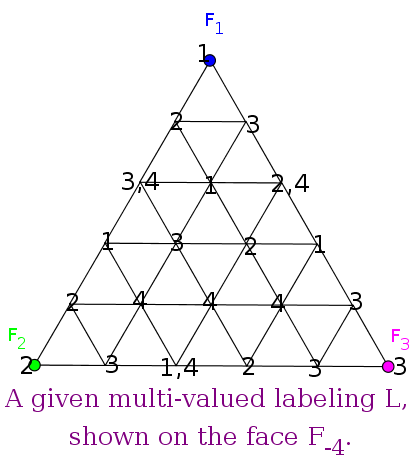}
\includegraphics[width=.21\textwidth]{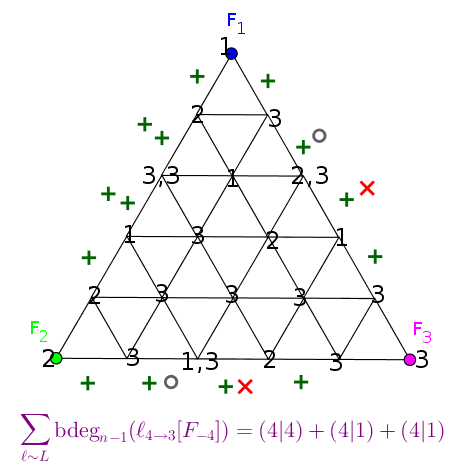}
\includegraphics[width=.21\textwidth]{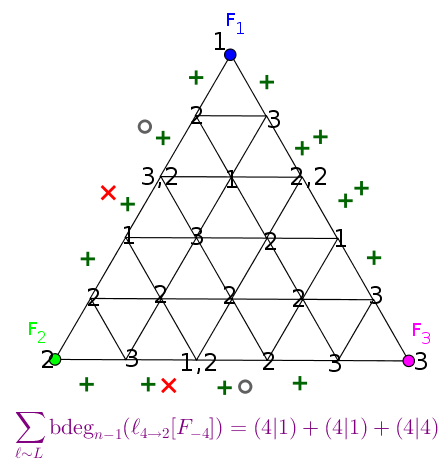}
\includegraphics[width=.21\textwidth]{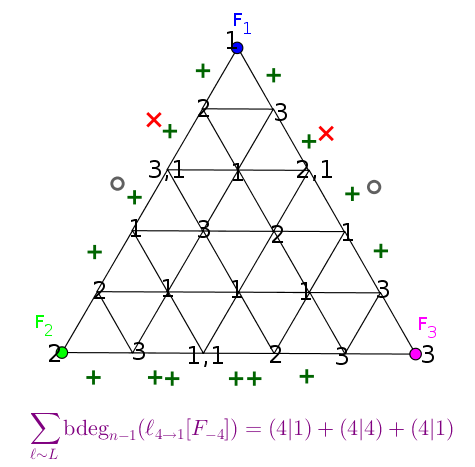}
\includegraphics[width=.08\textwidth]{4agents-reference.png}
\end{center}
\caption{
\label{fig:4agents-halfsperner}
An illustration of
Lemma \ref{lem:Kmodn}.
\\
\textbf{Left}: A consistent labeling $L$ on the face $F_{-4}$ of $\Delta^{4-1}$.
It has 3 negative vertices, labeled by $\{1,4\}$, $\{2,4\}$, $\{3,4\}$.
\\
\textbf{Middle}: 
The three labelings $\Lreplacen{j}$ for $j\in [3]$.
Each edge with a  ``+'' is positively oriented and contributes $+1$ to $\bdeg_{n-1}$; each edge with an ``X'' is negatively oriented and contributes $-1$; each edge with an ``O'' contributes 0. 
Two adjacent symbols correspond to the two options of choosing a label for the multi-labeled negative-vertex. 
On each edge there are two different sequences of labels induced by $L$.
\\
\textbf{Right}:
The reference simplex $Q$.
}
\end{figure*}

The case when $L$  cannot be reduced to a single-valued labeling is more difficult. The following lemma provides a partial treatment.
\begin{lemma}
\label{lem:Kmodn}
Let $L: \vrt(T)\to \multiset{n}$ be a consistent labeling of a friendly triangulation of $\Delta^{n-1}$.

Then, there exists an integer $K$ which is a product of integers from $[n-1]$, such that:
\begin{align*}
\sum_{\ell\sim L}
\bdeg(\ell)
\equiv K~\mod~n.
\end{align*}
\end{lemma}
\begin{proof}
As in previous proofs, we remove from $L$ as many multiple labels as possible while keeping $L$ consistent. 
Now, $L$ assigns multiple labels only to
\emph{negative vertices} --- vertices 
$x\in F_{[n]\setminus J}$ for which $L(x)=J$ --- when $|J|\geq 2$ (these vertices represent partitions where there are two or more empty pieces and the agent wants an empty piece).

By Lemma \ref{lem:sumreplace2}, the left-hand side equals $
 - 
 \sum_{j=1}^{n-1}
 \sum_{\ell\sim \Lreplacen{j}}
\bdeg_{n-1}(\ell [F_{-n}]) \mod{n}
$. We now calculate this sum. We focus on a specific $j$ and calculate
\begin{align*}
\sum_{\ell\sim \Lreplacen{j}}  \bdeg_{n-1}(\ell[F_{-n}]) \mod{n}.
\end{align*}

The boundary of the face $F_{-n}$ consists of $n-1$ faces. Let's denote them by $F_{-n,z}$ for $z\in[n-1]$. So:
\begin{align}
\label{eq:sumsum}
\sum_{\ell\sim \Lreplacen{j}} 
\bdeg_{n-1}(\ell[F_{-n}]) = 
\sum_{\ell\sim \Lreplacen{j}} 
\sum_{z\in [n-1]} \bdeg_{n-1}(\ell[F_{-n,z}])
\end{align}
Each face $F_{-n,z}$ has $n-2$ vertices and two zero coordinates --- $n$ and $z$ --- so negative vertices on that face are labeled with $\{n,z\}$.
When we replace $n$ by $j$, in each face $F_{-n,z}$, negative vertices are labeled with $\{j,z\}$.
Particularly, in $F_{-n,j}$, negative vertices are labeled twice with $j$, 
so the sum \eqref{eq:sumsum} counts the same labeling twice.
Denote the boundary-degree of that labeling on $F_{n,j}$ by $d_0$.

In the other $n-2$ faces $F_{-n,z}$ where $z\neq j$, the sum \eqref{eq:sumsum} counts two different labelings: 
\begin{itemize}
\item 
One labeling in which the negative vertex is labeled with $z$, which corresponds to a zero-coordinate of the face. By consistency and Lemma \ref{lem:constistency-degree}, its degree equals $d_0$ --- it is the same labeling up to a consistent permutation of the labels.
\item 
A second labeling in which the negative vertex is labeled with $j$, which 
corresponds to a non-zero coordinate of the face.
This labeling satisfies Sperner's condition. By Lemma \ref{lem:sperner1}, its degree equals $1$.
\end{itemize}
All in all, the sum \eqref{eq:sumsum} counts $n$ times $d_0$ plus $n-2$ times $1$.
The first term vanishes modulu $n$.
The second term should be multiplied by the number of different labels on each multi-labeled vertex outside that face. 
Each such vertex is labeled with at most $n-1$ different labels. 
All in all, the sum \eqref{eq:sumsum} is equal (modulu $n$) to $n-2$ times a product of integers in $[n-1]$.

By consistency, the sum \eqref{eq:sumsum} is the same for each $j\in[n-1]$. All in all, the sum $ \sum_{j=1}^{n-1}
\sum_{\ell\sim \Lreplacen{j}}
\bdeg_{n-1}(\ell[F_{-n}]) \mod{n}$ equals  $n-1$ times $n-2$ times a product of integers in $[n-1]$.
\end{proof}

An example for $n=4$ is illustrated in Figure \ref{fig:4agents-halfsperner}.
$F_{-4}$ has three edges. Each of these edges has a negative vertex with two possible labels.
In each of the three labelings $L_{4\to 3}$, $L_{4\to 2}$, $L_{4\to 1}$, 
each of the three edges has two possible labelings. Of the six labelings, four are identical (the labelings + + + +) so they contribute 0 to the sum of degrees modulu 4.
The other two labelings satisfy Sperner's condition (the labelings + O X +) so their degree is 1. Together they contribute 2 to the sum, but it should be multiplied by the number of options to choose labels on the other edges. All in all,
for each $j\in [3]$, the sum $\sum_{\ell\sim L_{4\to j}}\bdeg_{n-1}(\ell)$ equals $2$ times the number of options to choose labels on the other edges, which in this case is 4. 
Unfortunately, this makes the $\sum_{\ell\sim L_{4\to j}}\bdeg_{n-1}(\ell)$ equal zero modulu 4, so it does not help us prove Conjecture \ref{cnj:perm-deg}. 

Lemma \ref{lem:Kmodn} does imply Conjecture \ref{cnj:perm-deg} when $n$ is prime.
\begin{lemma}
\label{lem:prime}
Let $L: \vrt(T)\to 2^{[n]}$ be a consistent labeling of a friendly triangulation of $\Delta^{n-1}$, for some prime integer $n$.

Then, $L$ induces a single-valued labeling $\ell:\vrt(T)\to [n]$ with:
\begin{align*}
\bdeg(\ell)
\not\equiv 0 \mod{n}.
\end{align*}
\end{lemma}
\begin{proof}
The unique feature of a prime number $n$ is that it is not a multiple of integers from $[n-1]$. Therefore, Lemma \ref{lem:Kmodn} implies:
\begin{align*}
\sum_{\ell\sim L}
\bdeg(\ell)
\not\equiv 0 \mod{n}.
\end{align*}
Therefore, there exists at least one $\ell\sim L$ such that $\bdeg(\ell)
\not\equiv 0~\mod~n$.
\end{proof}

Our proof technique does not work when $n$ is not prime. Figure \ref{fig:4agents-halfsperner} shows that it fails even when $n=4$. 
This does not mean that the conjecture is false. For the conjecture, it is not necessary that the sum be nonzero --- it is sufficient that a single term in the sum be nonzero. There may be smarter ways of proving this than just taking a sum of all possible induced labelings. 
In fact, \citet{meunier2018envy} proved Conjecture \ref{cnj:perm-deg} for $n=4$ under an additional condition on the triangulation. The other cases are still open.
\fi

\section{Acknowledgments}
This work began when I was visiting the economics department in Glasgow university. I am grateful to Herve Moulin for the hospitality, guidance and inspiration.
Part of this work was done during my Ph.D. studies in Bar-Ilan university \citep{SegalHalevi2017Phd}.

The presentation and content of this article benefited a lot from comments by Oleg R. Musin, Peter Landweber, Ron Adin, participants in Bar-Ilan University combinatorics seminar and the St. Petersburg Fair Division conference, and three anonymous AAMAS reviewers.
I am grateful to Tahl Nowik, Francis E. Su, Frederic Meunier, Douglas Zare, Allen Hatcher, Kevin Walker, Francisco Santos and Lee Mosher for the mathematical tips.

\ifdefined\FULLVERSION
\newpage
\appendix
\section{Known Algorithms for Connected Envy-Freeness Do Not Work with Mixed Valuations}
\label{app:connected}
When there are at least 3 agents, connected envy-free cake-cutting cannot be found by a finite discrete procedure \citep{Stromquist2008Envyfree}. For 3 agents, several non-discrete procedures are known. These procedures use \emph{moving knives}.

\subsection{Rotating-knife procedure}
This beautiful procedure of \citet[pages 77-78]{Robertson1998CakeCutting} can be used only when the cake has at least two dimensions --- it cannot be used when the cake is a one-dimensional interval. However, this author has a special fondness for two-dimensional cakes \citep{SegalHalevi2017Fair,SegalHalevi2015EnvyFree} so he does consider this a pro rather than a con.

For simplicity assume that the cake is a convex 2-dimensional object, though the procedure can be extended to more general geometric settings. When all value-densities are positive, the procedure works as follows. 

Initially each agent marks a line parallel to the $y$ axis, such that the cake to the left of its line equals exactly $1/3$ by this agent's valuation. The leftmost mark is selected; suppose this mark belongs to Alice. Alice receives the piece to the left of her mark, and the remainder has to be divided among Bob and Carl.

Alice places a knife that divides the remainder into two pieces equal in her eyes. She rotates the knife slowly such that the two pieces at the two sides of the knife remain equal (this is possible to do for every angle). By the intermediate value theorem, there exists an angle such that Bob thinks that the two pieces at the two sides of the knife are equal too. At this point, Bob shouts ``stop'', the cake is cut, Carl picks the piece he prefers and Bob receives the last remaining piece.

For Alice, all three pieces have the same value, so she does not envy anyone; this is true even with mixed valuations. For Bob and Carl, the division of the remainder is like cut-and-choose so they do not envy each other; this too is true even with mixed valuations. When the valuations are positive, both Bob and Carl do not envy Alice, since her piece is contained in their leftmost $1/3$ pieces so it is worth for them less than $1/3$. However, this claim is true only when their value-densities are positive.

The procedure can be adapted to the case in which all value-densities are weakly-negative: in the first step, the \emph{rightmost} mark is selected instead of the leftmost one. However, with mixed valuations this adaptation does not work either. For example, suppose that the cake is piecewise-homogeneous with 4 homogeneous parts, and the agents' values to these parts are:

\begin{center}
\begin{tabular}{|c|c|c|c|c|}
\hline 
Alice: & -1 & 2 & 2 & -6 \\ 
\hline 
Bob: & 1 & -2 & 2 & -4 \\ 
\hline 
Carl: & 3  & -2 & -2 & -2 \\ 
\hline 
\end{tabular} 
\end{center}
For all agents, the entire cake is worth $-3$, so in the first step, each agent marks a line such that the cake to its left is worth $-1$. Thus Alice's mark is after the first slice to the left, Bob's mark is after the second slice and Carl's mark is after the third slice. Then:
\begin{itemize}
\item 
If Alice receives the piece to the left of her mark, then Bob might envy her even if he gets a half of the remainder;
\item If Bob receives the piece to the left of his mark, then Carl might envy him even if he gets a half of the remainder;
\item If Carl receives the piece to the left of his mark, then Alice might envy him even if she gets a half of the remainder.
\end{itemize}
So the procedure cannot be adapted, at least not in a straightforward way.

\subsection{Two-moving-knives procedure}
This procedure of \citet{Barbanel2004Cake} works also for a cake of one or more dimensions (to guarantee that the pieces are connected, it should be assumed that the cake is convex; all knives and cuts are parallel).

The first step is the same as in the rotating-knife procedure: the agents mark their $1/3$ line and the leftmost mark is selected; suppose this mark belongs to Alice. In the second step, Alice divides the remainder into two pieces equal in her eyes. Then there are three cases:
\begin{itemize}
\item If Bob prefers the middle piece and Carl the right piece or vice versa, then each of them gets his preferred piece and Alice gets the leftmost piece.
\item If both Bob and Carl prefer the middle piece, then Alice holds two knives at the two ends of the middle piece and moves them inwards, keeping the two external pieces equal in her eyes. When either Bob or Carl believes that the middle piece is equal to one of the external pieces, he shouts ``stop'' and takes that external piece. The non-shouter takes the middle piece and Alice takes the other external piece.
\item If both Bob and Carl prefer the rightmost piece, then Alice holds two knives at the two ends of the middle piece and moves them \emph{rightwards}, keeping the two \emph{leftmost} pieces equal in her eyes; then the procedure proceeds as in the previous case.
\end{itemize}
When all valuations are positive, these are the only possible cases, since both Bob and Carl believe that Alice's piece is worth at most $1/3$. The procedure can be adapted to the case of all-negative valuations, by putting the two knives in the hand of the \emph{rightmost} cutter. However, this adaptation does not work with mixed valuations, as shown by the example in the previous subsection.

\subsection{Four-moving-knives procedure}
This procedure of \citet{Stromquist1980How} was the first procedure for connected envy-free division. It requires a ``sword'' moved by a referee, and three knives moved simultaneously by the three agents. It works for a convex cake in one or more dimensions; again all knives and cuts are parallel.

The sword moves constantly from the left end of the cake to its right end. Each agent holds his knife in a point that divides the cake into the right of the sword to two pieces equal in his eyes. The first agent that thinks that the leftmost piece is sufficiently valuable (equal to the piece at the left/right of the middle knife) shouts ``stop'' and receives the leftmost piece. Then, the middle knife cuts the remainder and each of the non-shouters gets a piece that contains its knife.

The correctness of this procedure depends on the assumption that the non-shouters will not envy the shouter (since otherwise they should have shouted earlier). However, this is true only if the piece to the left of the sword grows monotonically as the sword moves rightwards. When the valuations are mixed, the monotonicity breaks down, and with it, the no-envy guarantee.

\subsection{Approximation algorithms}
For additive agents, \citet{Branzei2017Query} present a general procedure for finding an $\epsilon$-approximation for \emph{any} condition described by linear constraints. 
Whenever there exists an allocation that satisfies such a condition, their algorithm finds an allocation in which the value of each agent is at most $\epsilon$ less than its required value.  
In particular, whenever an envy-free allocation exists, 
their algorithm finds an allocation in which each agent values its piece as at most $\epsilon$ less than the piece of any other agent (the agents' valuations are normalized such that the entire cake-value is $1$ for all agents, so $\epsilon$ is a fraction, e.g., $0.01$ of the entire cake value). Their algorithm works as follows:
\begin{enumerate}
\item Each agent makes several marks on the cake, such that its value for the piece between each two consecutive marks is at most $\epsilon$.
\item The algorithm checks all combinations of $n-1$ marks; each such combination defines a connected division. If an envy-free division exists, then necessarily one of the checked divisions represents an $\epsilon$-envy-free division.
\end{enumerate}
This algorithm works well for mixed cakes. In particular, for $n=3$, an envy-free allocation exists, so an $\epsilon$-envy-free allocation will be found by the above procedure.

However, there is a ``catch''. When the cake is good, the number of queries required is $O(n/\epsilon)$, since each agent has to make $O(1/\epsilon)$ marks.
This is also true when the cake is bad; in this case, the values between each two consecutive marks will be $-\epsilon$. However, when the cake is mixed, the number of marks might be arbitrarily large: each agent might have an unbounded number of $+\epsilon$ and $-\epsilon$ pieces. 

For the case $n=3$, \citet{Branzei2017Query} present an improved approximation algorithm that finds an $\epsilon$-envy-free allocation in $O(\log(1/\epsilon))$ queries. However, this algorithm approximates the Barbanel-Brams two-knives procedure, which does not work with mixed cakes (see above).

Therefore, the query complexity of finding an $\epsilon$-envy-free allocation in a mixed cake remains an open question.

\section{Envy-freeness with disconnected pieces}
\label{app:disconnected}
Without the connectivity requirement,
more options for envy-free division are available.

\subsection{Exactly-equal and nearly-exactly-equal divisions}
Suppose the value-densities of the agents are normalized such that each agent values the entire cake as 1. Then, it is possible to divide the cake into pieces each of which is worth \emph{exactly} $1/n$ for every agent. Such a division is envy-free whether the valuations are positive, negative or mixed.
The existence of such partitions was proved by \citet{Dubins1961How}; later, \citet{Alon1987Splitting} proved it can be done with a bounded number of cuts. However, this number is still much larger than $n-1$, so the pieces will not necessarily be connected. In fact, \citet{Alon1987Splitting} showed a simple example in which it is impossible to have an exactly-equal division with connected pieces, even with positive valuations, let alone mixed valuations.

Even without connectivity, it is impossible to find an exactly-equal division with a finite number of queries.
The algorithm of \citet{Robertson1998CakeCutting} uses a finite number of queries to find a nearly-exactly-equal division, which is also envy-free. At first glance, it seems this algorithm should work for mixed valuations too, but the details require more work.

\subsection{Trimming and enlarging}
The first algorithm for envy-free division for three agents was devised by Selfridge and Conway \citep[pages 116-120]{Brams1996Fair}. It introduced the idea of \emph{trimming}. Let Alice cut the cake into three pieces equal in her eyes. Then ask Bob and Carl which piece they prefer. If they prefer different pieces then we are done. If the prefer the same piece, then let Bob trim this best piece so that it's equal to his second-best piece. Now, Carl takes any piece he wants, Bob takes one of his two best pieces (at least one of these remains on the table), and Alice takes one of her three original pieces (at least one of these remains on the table). We have an envy-free division of a part of the cake; the trimmings remain on the table and are divided by a second step, which we skip here for brevity.

The idea of trimming a best piece to make it equal to the second-best piece lies at the heart of more sophisticated algorithms for $n$ agents, such as \citet{Brams1995EnvyFree} and \citet{Aziz2016Discrete}. This idea crucially relies on all valuations being  positive, so that trimming a piece makes it weakly less valuable for all agents. 

When all valuations are negative, the analogue of trimming is enlarging --- the smallest piece should be enlarged to make it equal to the second-smallest; however, it is not immediately clear how this enlargement can be done --- where should the extra cake come from? The first solution was devised by Reza Oskui \citep[pages 73-75]{Robertson1998CakeCutting} for three agents. 
The idea of enlarging pieces was further developed by \citet{Peterson2009Nperson}, who presented an algorithm for $n$ agents. Their algorithm is discrete and requires a finite, but unbounded, number of queries. 

When valuations are mixed, trimming or enlarging a piece can make it better for some players and worse for some other players. Therefore, it is not clear how any of these procedures can be adapted.

\subsection{Dividing positive and negative parts separately}
There is another simple trick that can be used when there is no connectivity requirement. The idea is to divide the cake into sub-cakes of two types:
\begin{enumerate}
\item Sub-cakes whose value is positive for at least one agent;
\item Sub-cakes whose value is negative for all agents.
\end{enumerate}
Sub-cakes of the first kind should be divided among the agents who value them positively, using any algorithm for envy-free division with positive valuations; sub-cakes of the second kind should be divided among all $n$ agents, using any algorithm for envy-free division with negative valuations.

This algorithm can be done in finite time if and only if, for each agent, the cake can be divided to a finite number of pieces, each of which is entirely-positive or entirely-negative (in other words, the number of switches between positive and negative value-density is finite for every agent). 

Even with this condition, the algorithm does not fit the standard Robertson-Webb query model. This model allows to ask an agent to mark a piece of cake having a certain value, but there is no query of the form ``mark the cake at a point where your value switches between negative and positive''. 

Still, in practice this algorithm seems like the most reasonable alternative: it does not make sense to give a cake to an agent who thinks it is bad, when other agents think it is good.

\fi
\bibliographystyle{ACM-Reference-Format}  
\balance
\bibliography{../erelsegal-halevi}
\end{document}